\documentclass[11pt,authoryear]{elsarticle}

\usepackage{amsmath,amsfonts,amssymb,amsthm,color,graphicx}

\usepackage{mathtools} % for \vcentcolon to define := and =:
\usepackage{appendix} % to get the appendix
\usepackage{dsfont} % to get \mathds for indicator functions
\usepackage{booktabs} % more options with tables
\usepackage{enumitem} % to reference items in itemize lists
\usepackage{url} % to break up long url links with line breaks
\usepackage{float} % better placement of tables
\usepackage{xurl} % flexible break points for urls
\usepackage{hyperref} % hyperlinks
\usepackage{geometry} % controls the margins
\newtheorem{theorem}{Theorem}
\newtheorem{proposition}[theorem]{Proposition}

\newtheorem{corollary}[theorem]{Corollary}
\newtheorem{remark}{Remark}

\numberwithin{equation}{section}

\newcommand{\N}{\mathbb{N}}
\newcommand{\R}{\mathbb{R}}
\newcommand{\PP}{\mathsf{P}} % Russian style (don't change)
\newcommand{\EE}{\mathsf{E}} % Russian style (don't change)
\newcommand{\Bias}{\mathsf{Bias}} % Russian style (don't change)
\newcommand{\Var}{\mathsf{Var}} % Russian style (don't change)
\newcommand{\Cov}{\mathsf{Cov}} % Russian style (don't change)
\newcommand{\bb}[1]{\boldsymbol{#1}}
\newcommand{\OO}{\mathcal{O}}
\newcommand{\oo}{\mathrm{o}}
\newcommand{\rd}{\mathrm{d}}
\newcommand{\ind}{\mathds{1}}
\newcommand{\leqdef}{\vcentcolon=}
\newcommand{\reqdef}{=\vcentcolon}

\allowdisplaybreaks

\begin{document}

\begin{frontmatter}

\title{A Bernstein polynomial approach for the estimation of cumulative distribution functions in the presence of missing data}

\author[a1]{Rihab Gharbi}\ead{rihab.gharbi@fst.utm.tn}
\author[a2]{Wissem Jedidi}\ead{wjedidi@ksu.edu.sa}
\author[a1]{Salah Khardani}\ead{salah.khardani@fst.utm.tn}
\author[a3]{Fr\'ed\'eric Ouimet\corref{mycorrespondingauthor}}\ead{frederic.ouimet2@uqtr.ca}

\address[a1]{Laboratoire des \'equations aux d\'eriv\'ees partielles, Universit\'e de Tunis El-Manar, Tunisia}
\address[a2]{Department of Statistics and Operations Research, College of Science, King Saud University, Saudi Arabia}
\address[a3]{D\'epartement de math\'ematiques et d'informatique, Universit\'e du Qu\'ebec \`a Trois-Rivi\`eres, Canada\vspace{-6mm}}

\cortext[mycorrespondingauthor]{Corresponding author. Email address: frederic.ouimet2@uqtr.ca}

\begin{abstract}
We study nonparametric estimation of univariate cumulative distribution functions (CDFs) pertaining to data missing at random. The proposed estimators smooth the inverse probability weighted (IPW) empirical CDF with the Bernstein operator, yielding monotone, $[0,1]$-valued curves that automatically adapt to bounded supports. We analyze two versions: a pseudo estimator that uses known propensities and a feasible estimator that uses propensities estimated nonparametrically from discrete auxiliary variables, the latter scenario being much more common in practice. For both, we derive pointwise bias and variance expansions, establish the optimal polynomial degree $m$ with respect to the mean integrated squared error, and prove the asymptotic normality. A key finding is that the feasible estimator has a smaller variance than the pseudo estimator by an explicit nonnegative correction term. We also develop an efficient degree selection procedure via least-squares cross-validation. Monte Carlo experiments show that, for small to moderate sample sizes, the Bernstein-smoothed pseudo and feasible estimators outperform their unsmoothed counterparts and the integrated version of the IPW kernel density estimator proposed by \citet{MR2569798}, under certain models. A real-data application to fasting plasma glucose from the 2017--2018 NHANES survey illustrates the method in a practical setting. All code needed to reproduce our analyses is readily accessible on GitHub.
\end{abstract}


\end{frontmatter}

\vspace{-5mm}
\section{Introduction}

The cumulative distribution function (CDF) and its inverse (the quantile function) are fundamental objects in probability and statistics, providing a complete description of a random variable's distribution. They offer insights into the entire distribution, including tail behavior, that summary measures like the mean or variance cannot capture. Accurate estimation of CDFs is crucial across numerous fields, for example, in survival analysis (where the CDF relates to survival probabilities), reliability engineering (where the CDF gives failure probabilities and reliability over time), economics (where the CDF characterizes income, wealth, or return distributions), and risk management (where the loss CDF underpins tail-risk metrics such as Value-at-Risk), because one often needs to understand the full distribution rather than just a single parameter.

In practice, statistical analysis is frequently complicated by incomplete data. Missing observations are ubiquitous in large studies, clinical trials, and surveys. If data are missing in a nontrivial fashion, standard nonparametric estimators such as the empirical CDF computed from observed cases can be biased for the true distribution unless the data are missing completely at random, that is, the missingness mechanism is independent of the data values \citep[p.14]{LittleRubin2019}. A more plausible assumption in many contexts is missing at random (MAR), which posits that the probability that an observation is missing may depend on fully observed auxiliary variables but not on the value of the missing observation itself \citep{MR455196}. Under MAR, one can obtain unbiased estimates by appropriately reweighting or imputing the observed data. A prominent approach is inverse probability weighting (IPW), inspired by the Horvitz--Thompson estimator from survey sampling \citep{MR53460}. In IPW, each observed case is weighted by the inverse of its observation probability (called the propensity score), often estimated from covariates \citep{MR742974}, yielding a weighted empirical CDF (sometimes called a pseudo-empirical CDF). When propensity scores are known or consistently estimated, this IPW empirical CDF is asymptotically unbiased for the full-population CDF.

Alternative nonparametric strategies for MAR data include imputation-based methods and hybrid approaches. For instance, \citet{MR1379049} proposed a kernel regression imputation to consistently estimate the CDF and quantiles, proving strong uniform consistency and asymptotic normality. Hybrid strategies, in particular augmented IPW estimators, have been developed to improve efficiency by combining weighting with outcome imputation. \citet{MR2644095} introduced an augmented IPW empirical-likelihood approach for CDFs with missing responses that delivers asymptotic Gaussian limits for the process. These methods underscore that, under MAR assumptions, one can recover distributional information using either weighting or imputation, or both, without fully specifying outcome models. Furthermore, these estimation strategies have also been adapted to the more complex case of nonignorable missing data (NMAR). Related work by \citet{MR3782668} considered IPW, regression imputation, and an augmented approach for distribution and quantile estimation under a semiparametric NMAR model, finding that all three can reach the same asymptotic variance under suitable conditions.

However, the IPW-based empirical CDF, like the ordinary empirical CDF, is a step function. When the underlying distribution is continuous, it is often desirable to smooth such step-function estimates to improve efficiency, reduce mean squared error (MSE), and produce a smooth, continuous CDF estimate. Smoothing can yield substantial gains; some smooth CDF estimators have been shown to asymptotically outperform the raw empirical CDF in mean integrated squared error (MISE). Kernel smoothing is a popular approach in the density estimation context; see, e.g., \citet{MR79873,MR143282}, or \citet{MR848134,MR1319818} for book treatments. This approach extends naturally to CDFs by integrating the kernel density estimate, resulting in the classical smooth CDF estimator introduced by \citet{MR166862}, which replaces the indicator steps of the empirical CDF with integrated kernels. An alternative strategy involves direct regression smoothing of the empirical CDF. For example, \citet{MR1947061} proposed a local linear estimator that can achieve a smaller asymptotic MISE than the integrated kernel estimator, although it does not guarantee monotonicity. These kernel methods have been adapted to handle missing-data scenarios; for example, \citet{MR2569798} developed an IPW kernel density estimator (KDE) under MAR. For CDFs under MAR, a natural idea is to smooth the IPW variant using similar techniques. For instance, one of the competitors considered in our Monte Carlo study (Section~\ref{sec:simulations}) is an integrated version of Dubnicka's estimator.

Smoothing the CDF can improve pointwise variance and yield well-behaved estimates, but one must address the well-known boundary bias that traditional kernel estimators exhibit on bounded supports; see, e.g., \citet{doi:10.17713/ajs.v49i1.801}. If the support of the distribution is [0,1], integrated-KDE CDF estimators tend to incur bias near 0 and 1 due to a spill-over effect created by the fixed kernel. To mitigate this, various boundary-correction techniques have been proposed, including reflecting the data at boundaries or using boundary kernels tailored for the edges \citep{MR3072469}. A user-friendly and highly effective strategy employs asymmetric kernels, which inherently respect support limits and ensure monotonicity by construction; see, e.g., \citet{doi:10.3390/math9202605,jsci2024349}. Alternatively, monotonicity constraints can be imposed on smoothing techniques that do not naturally guarantee it. For instance, \citet{MR2662606} proposed a constrained polynomial spline approach that smooths the empirical CDF while enforcing monotonicity, thereby guaranteeing a valid CDF estimate.

In this paper, the focus is on CDFs supported on the unit interval $[0,1]$ (more general bounded supports can often be transformed to $[0,1]$). For such cases, Bernstein polynomials offer an elegant and effective smoothing technique. Bernstein polynomials \citep{Bernstein_1912} were originally introduced as a constructive proof of the Weierstrass approximation theorem \citep{Weierstrass_1885}, which asserts the existence of uniform polynomial approximations for continuous functions over a closed interval. When used for smoothing an empirical CDF, Bernstein operators \citep[see, e.g.,][]{MR3585481} produce estimates that automatically respect the boundaries, mapping $[0,1]$ into $[0,1]$, and are genuine CDFs, since they are monotone non-decreasing and bounded between 0 and 1 by construction. They also enjoy shape-preserving properties, meaning that the smoothed curve inherits the qualitative shape constraints of a CDF. The idea of leveraging Bernstein polynomials for nonparametric estimation dates back to \citet{MR0397977}, who studied a Bernstein estimator for density functions on a bounded interval. \citet{MR1910059} and \citet{MR2270097} later demonstrated the utility of Bernstein polynomials for CDF estimation, applying them to smooth the empirical CDF and the associated normalized histogram under strongly mixing data. Subsequent work has analyzed the statistical properties of Bernstein estimators in detail. \citet{MR2960952,MR2925964} derived higher-order bias and variance expansions for the Bernstein CDF estimator and showed that it has asymptotically negligible boundary bias and can be more efficient than the (unsmoothed) empirical CDF. There have also been several developments and extensions of Bernstein-type smoothers. For example, \citet{MR3740720} and \citet{MR4365177} developed and analyzed a recursive approach to Bernstein CDF estimation, including a Robbins--Monro style estimator and moderate deviation theory. \citet{MR3899096} introduced an alternative CDF estimator using rational Bernstein polynomials, which generalizes the classical Bernstein approach to improve flexibility. For distributions supported on $[0,\infty)$, \citet{MR4330318} proposed using Szasz--Mirakyan positive linear operators, an analog of Bernstein operators for $[0,\infty)$, to construct smooth CDF estimators for lifetime or survival data, and they showed that this method outperforms the empirical CDF in MSE and MISE as well. Moreover, Bernstein polynomials have been adapted to other complex settings. For example, \citet{MR3630225} estimated conditional CDFs by smoothing regression estimators, and \citet{MR4765928} extended Bernstein-polynomial CDF and quantile estimation to right-censored data, highlighting their broad applicability. All of these approaches share a common goal of producing a smoother and more efficient estimate of the CDF while enforcing the shape constraints that are intrinsic to CDFs.

Despite the rich literature on handling missing data and on smooth CDF estimation separately, relatively little attention has been given to smoothing CDF estimators in the presence of missing data. This work aims to fill that gap by developing and analyzing a nonparametric CDF estimator for MAR data that marries the IPW principle with Bernstein polynomial smoothing. In other words, the IPW-adjusted empirical CDF, which corrects bias due to MAR missingness, is smoothed using the Bernstein operator (see \eqref{eq:Bernstein.operator} below) to obtain a smooth and shape-constrained estimate. This approach inherits two benefits: unbiasedness under MAR from the weighting and reduced variance plus automatic boundary adaptation from the Bernstein smoothing.

The remainder of the paper is organized as follows. Section~\ref{sec:setup} introduces the statistical framework, defines the Bernstein operator, formulates the MAR setting, describes propensity-score estimation from discrete covariates, and constructs the Bernstein-smoothed IPW estimators $\smash{\widetilde{F}_{n,m}}$ (when propensities are known) and $\smash{\widehat{F}_{n,m}}$ (when propensities are estimated). Section~\ref{sec:results} presents the results: pointwise bias and variance, optimal choices of $m$ based on MSE and MISE, and asymptotic normality. Results are given first for the pseudo estimator with known propensities (Section~\ref{sec:asymp.tilde.F.n.m}) and then for the feasible estimator with estimated propensities (Section~\ref{sec:asymp.hat.F.n.m}), where we quantify the variance reduction from estimating the propensity scores. Section~\ref{sec:simulations} reports a Monte Carlo study comparing the smoothed and unsmoothed estimators across sample sizes, including the integrated-IPW KDE of \citet{MR2569798} as a benchmark. Section~\ref{sec:application} applies the feasible Bernstein-smoothed estimator to a fasting plasma glucose dataset. Section~\ref{sec:summary.outlook} summarizes the contributions and outlines directions for future research. All proofs are collected in Section~\ref{sec:proofs}. For reproducibility, Appendix~\ref{app:code} links to the \textsf{R} code used to generate the figures, the simulation results, and the real-data application. Appendix~\ref{app:acronyms} provides a list of acronyms used throughout.

\section{The setup}\label{sec:setup}

For any continuous function $\varphi:[0,1]\mapsto \R$, the Bernstein operator,
\begin{equation}\label{eq:Bernstein.operator}
\mathcal{B}_m(\varphi)(y) = \sum_{k=0}^m \varphi(k/m) \binom{m}{k} y^k (1 - y)^{m-k}, \quad y\in [0,1], ~~m\in \N,
\end{equation}
defines a sequence of bounded polynomials, $\{\mathcal{B}_m(\varphi)\}_{m\in \N}$, which converges uniformly to $\varphi$. The weight function that smooths out the discrete values of $\varphi$ over the mesh $\{0,1/m,2/m,\ldots,1\}$ corresponds to the probability mass function of a binomial distribution as a function of its success probability parameter $y$. This weight function is denoted, for all $y\in [0,1]$ and $k\in \{0,\ldots,m\}$, by
\[
b_{m,k}(y) = \binom{m}{k} y^k (1 - y)^{m-k}.
\]

\begin{remark}\label{rem:rescaling}
The Bernstein operator in \eqref{eq:Bernstein.operator} is defined for functions on $[0,1]$, so our theory is stated on that support. This is not restrictive for bounded responses: if a variable $Z$ is supported on $[a,b]$ with $a<b$, then the linearly rescaled variable $U=(Z-a)/(b-a)\in[0,1]$ can be analyzed by our method, and the original CDF is recovered by $F_Z(z)=F_U((z-a)/(b-a))$ for $z\in[a,b]$.
\end{remark}

Consider the following setting. In the population under study, there are a continuous response variable $Y$ subject to missingness and a discrete auxiliary variable $\bb{X} = (X_1,\ldots,X_d)$, for some $d\in \N$, which is fully observed. This is a common scenario in practice, for example with categorical demographics in surveys. Take a random sample of $n$ independent and identically distributed (i.i.d.) pairs $(\bb{X}_1,Y_1),(\bb{X}_2,Y_2),\ldots,(\bb{X}_n,Y_n)$ from the population. For each unit $i\in \{1,\ldots,n\}$ in the sample, let $Y_i$ and $\bb{X}_i = (X_{i1}, \ldots,X_{id})$ denote the $i$th response and $i$th covariate, respectively. Let the Bernoulli random variable
\[
\delta_i
= \ind_{\{Y_i ~\text{is observed}\}}
=
\begin{cases}
1, &\mbox{if } Y_i~\text{is observed}, \\
0, &\mbox{otherwise},
\end{cases}
\]
indicate whether $Y_i$ is observed or not. We assume that the distribution of $\delta_i$ given $\bb{X}_i$ is defined by the {\it propensity score}:
\[
\pi(\bb{X}_i) = \PP\left(\delta_i = 1 \mid Y_i, \bb{X}_i\right) = \PP\left(\delta_i=1\mid \bb{X}_i\right).
\]
Note that $\delta_i$ is independent of $Y_i$ conditionally on $\bb{X}_i$.

The goal is to estimate the unknown CDF $F$ of the observations $Y_1,\ldots,Y_n$, which are MAR. We investigate two versions of the estimator based on the knowledge of the propensity scores.

The first version assumes that propensities are known, which is the case, for example, when the missingness mechanism is designed. In this case, $F$ can be estimated by the following IPW pseudo-empirical estimator:
\begin{equation}\label{eq:tilde.F.n}
\widetilde{F}_n(y) = n^{-1} \sum_{i=1}^n \frac{\delta_i}{\pi(\bb{X}_i)} \ind_{\{Y_i \leq y\}}.
\end{equation}

The second and more realistic version addresses the fact that, in practice, the propensity scores are usually unknown. In this paper, we restrict attention to auxiliary variables $\bb{X}$ with discrete finite support. This choice is mainly for expositional convenience: it avoids the tedious task of treating the discrete and continuous cases in parallel. In the discrete case, propensity score estimation for MAR adjustment is simplified, allowing us to rely on the following nonparametric estimator:
\begin{equation}\label{eq:hat.pi}
\hat{\pi}_i(\bb{X}_{1:n}) = \frac{\sum_{j=1}^n \delta_j \, \ind_{\{\bb{X}_j = \bb{X}_i\}}}{\sum_{k=1}^n \ind_{\{\bb{X}_k = \bb{X}_i\}}}.
\end{equation}
If the auxiliary variable were continuous, one would simply replace \eqref{eq:hat.pi} by a Nadaraya--Watson estimator; see \citet[Eq.~(4)]{MR2569798}. The main asymptotic results stated in Section~\ref{sec:results} would remain the same, but presenting both settings side by side would add substantial technical bookkeeping. For this reason, we formulate the theory only for the discrete case. The CDF $F$ can then be estimated using the feasible empirical estimator:
\[
\widehat{F}_{n}(y)= n^{-1} \sum_{i=1}^n \frac{\delta_i}{\hat{\pi}_i(\bb{X}_{1:n})} \ind_{\{Y_i \leq y\}}.
\]

The aim of this paper is to show that we can smooth $\widetilde{F}_n$ and $\widehat{F}_n$ using Bernstein polynomials to improve their performance. Applying the Bernstein operator yields an oracle or pseudo estimator $\smash{\widetilde{F}_{n,m}}$ when propensities are known, and a feasible estimator $\smash{\widehat{F}_{n,m}}$ when propensities are estimated nonparametrically from the data. More specifically, define, for all $y\in [0,1]$ and $m\in \N$,
\begin{equation}\label{eq:estimators}
\begin{aligned}
\widetilde{F}_{n,m}(y)
&= \mathcal{B}_m(\widetilde{F}_n)(y)
= \sum_{k=0}^m \widetilde{F}_n(k/m) \, b_{m,k}(y)
= n^{-1} \sum_{k=0}^m \sum_{i=1}^n \frac{\delta_i}{\pi(\bb{X}_i)} \ind_{\{Y_i \leq k/m\}} b_{m,k}(y), \\
%%%
\widehat{F}_{n,m}(y)
&= \mathcal{B}_m(\widehat{F}_n)(y)
= \sum_{k=0}^m \widehat{F}_n(k/m) \, b_{m,k}(y)
= n^{-1} \sum_{k=0}^m \sum_{i=1}^n \frac{\delta_i}{\hat{\pi}_i(\bb{X}_{1:n})} \ind_{\{Y_i \leq k/m\}} b_{m,k}(y).
\end{aligned}
\end{equation}

\section{Results}\label{sec:results}

In this section, we present theoretical properties of the proposed Bernstein estimators. We derive asymptotic expansions for the pointwise bias and variance, determine the optimal polynomial degree $m$ that minimizes the MSE and MISE, and establish asymptotic normality for both the pseudo and feasible estimators defined in \eqref{eq:estimators}.

\paragraph{Assumptions}
Beyond the setup presented in Section~\ref{sec:setup}, we assume the following:
\begin{enumerate}[label=(A\arabic*)]\setlength\itemsep{0em}
\item The propensity scores are bounded from below, i.e., $\pi_{\min} \leqdef \min_{\bb{x}} \pi(\bb{x}) > 0$. \label{ass:1}
\item The distribution of $\bb{X}$ is supported on finitely many values, so that $p_{\min} \leqdef \min_{\bb{x}} p(\bb{x}) > 0$. \label{ass:2}
\item The CDF $F$ (equivalently, $F_{Y_1}$) is three times continuously differentiable on $(0,1)$. \label{ass:3}
\item The CDF $F_{Y_1 \mid \bb{X}_1}$ is twice continuously differentiable on $(0,1)$. \label{ass:4}
\end{enumerate}

\paragraph{Notation}
Throughout the paper, $f$ (equivalently, $f_{Y_1}$) denotes the density of $Y_1$, and similarly $\smash{f_{Y_1 \mid \bb{X}_1}}$ denotes the conditional density of $Y_1$ given $\bb{X}_1$. Since $Y_1$ is a continuous random variable, note that $f = F'$ and $\smash{f_{Y_1 \mid \bb{X}_1} = F_{Y_1 \mid \bb{X}_1}'}$. The degree parameter $m = m(n)$ is always assumed to be a function of the sample size $n$, and further $m\to \infty$ whenever $n\to \infty$. For real sequences $(U_n)$ and $(V_n)$, the notation $U_n = \OO(V_n)$ means that $\limsup_{n\to \infty} |U_n / V_n| \leq C < \infty$, where the nonnegative~constant~$C$ may depend on $d$, $\pi_{\min}$, $p_{\min}$, $F$, or $F_{Y_1 \mid \bb{X}_1}$, but no other variable unless explicitly written as a subscript. Typically, the dependence is pointwise on the variable $y$, in which case one writes $U_n = \OO_y(V_n)$. More specifically, $U_n = \OO_y(V_n)$ means that for any fixed $y$, there exists $C_y > 0$ such that $\limsup_{n\to \infty} |U_n / V_n| \leq C_y < \infty$. In some places, the Vinogradov notation $U_n \ll V_n$ is used to indicate $U_n = \OO(V_n)$ more concisely. Similarly, the notation $U_n = \oo(V_n)$ means that $\lim_{n\to \infty} |U_n / V_n| = 0$. Subscripts indicate which parameters the convergence rate can depend on. More generally, if $U_n,V_n$ are random variables, we write $U_n = \OO_{L^2}(V_n)$ if $\limsup_{n\to \infty} \EE[|U_n / V_n|^2] \leq C < \infty$ and $U_n = \oo_{L^2}(V_n)$ if $\lim_{n\to \infty} \EE[|U_n / V_n|^2] = 0$. If the rate depends on a variable such as $y$, we add a subscript.

\subsection{Asymptotic properties of \texorpdfstring{$\smash{\widetilde{F}_{n,m}}$}{tilde{F}\_n,m}}\label{sec:asymp.tilde.F.n.m}

We begin by analyzing the pseudo estimator $\smash{\widetilde{F}_{n,m}}$, which utilizes known propensity scores. This serves as a benchmark for the feasible estimator analyzed later.

The following proposition establishes the asymptotics of the pointwise bias. Since the underlying pseudo-empirical estimator $\widetilde{F}_n$ is unbiased for $F$, the bias of the smoothed estimator $\smash{\widetilde{F}_{n,m}}$ is entirely attributable to the Bernstein smoothing process, confirming the standard bias rate of order $m^{-1}$.

\begin{proposition}[Bias]\label{prop:bias.tilde.F.n.m}
Suppose that Assumptions~\ref{ass:1}~and~\ref{ass:3} hold. Then, for any $y\in (0,1)$, we have, as $n\to \infty$,
\[
\Bias(\widetilde{F}_{n,m}(y)) = \EE[\widetilde{F}_{n,m}(y)] - F(y) = m^{-1} B(y) + \oo_y(m^{-1}),
\]
where
\begin{equation}\label{eq:def.B}
B(y) = \frac{1}{2} y (1 - y) f'(y).
\end{equation}
\end{proposition}

Next, we examine the pointwise variance of the pseudo estimator. The expansion reveals two main components: the leading term $n^{-1} \sigma^2(y)$, corresponding to the variance of the unsmoothed IPW empirical CDF, and a negative correction term of order $n^{-1} m^{-1/2}$. This correction term explicitly quantifies the variance reduction achieved by the Bernstein smoothing.

\begin{proposition}[Variance]\label{prop:variance.tilde.F.n.m}
Suppose that Assumptions~\ref{ass:1}, \ref{ass:3}, and \ref{ass:4} hold. Then, for any $y\in (0,1)$, we have, as $n\to \infty$,
\[
\Var(\widetilde{F}_{n,m}(y)) = n^{-1} \sigma^2(y) - n^{-1} m^{-1/2} V(y) + \oo_y(n^{-1} m^{-1/2}),
\]
where
\begin{equation}\label{eq:def.sigma2.V}
\sigma^2(y) = \EE\left[\frac{F_{Y_1 \mid \bb{X}_1}(y)}{\pi(\bb{X}_1)}\right] - \{F(y)\}^2, \qquad
V(y) = \sqrt{\frac{y(1-y)}{\pi}} \EE\left[\frac{f_{Y_1 \mid \bb{X}_1}(y)}{\pi(\bb{X}_1)}\right].
\end{equation}
\end{proposition}

The MSE is an immediate consequence of the bias and variance expansions derived above. Furthermore, analyzing the MSE allows us to determine the optimal polynomial degree $m_{\mathrm{opt}}(y)$ that balances the trade-off between the squared bias term (of order $m^{-2}$) and the variance reduction term (of order $n^{-1}m^{-1/2}$).

\begin{corollary}[Mean squared error]\label{cor:MSE.tilde.F.n.m}
Suppose that Assumptions~\ref{ass:1}, \ref{ass:3}, and \ref{ass:4} hold. Then, for any $y\in (0,1)$, we have, as $n\to \infty$,
\[
\begin{aligned}
\mathrm{MSE}[\widetilde{F}_{n,m}(y)]
&= \Var(\widetilde{F}_{n,m}(y)) + \{\Bias(\widetilde{F}_{n,m})\}^2 \\
&= n^{-1} \sigma^2(y) - n^{-1} m^{-1/2} V(y) + m^{-2} B^2(y) + \oo_y(n^{-1} m^{-1/2}) + \oo_y(m^{-2}).
\end{aligned}
\]
In particular, if $B(y) V(y) \neq 0$, the asymptotically optimal choice of $m$, with respect to the $\mathrm{MSE}$, is
\[
m_{\mathrm{opt}}(y) = n^{2/3} \left[4 B^2(y)/V(y)\right]^{2/3},
\]
in which case, as $n\to \infty$,
\[
\mathrm{MSE}[\widetilde{F}_{m_{\mathrm{opt}},n}(y)] = n^{-1} \sigma^2(y) - n^{-4/3} \left[27 V^4(y)/\{256 B^2(y)\}\right]^{1/3} + \oo_y(n^{-4/3}).
\]
\end{corollary}

To obtain a global measure of accuracy, we integrate the MSE over $(0,1)$. This yields the MISE and allows for the determination of a globally optimal degree $m_{\mathrm{opt}}$. This optimal degree achieves a convergence rate of $n^{-4/3}$ for the second-order term, demonstrating a clear improvement over the unsmoothed estimator $\smash{\widetilde{F}_n}$.

\begin{corollary}[Mean integrated squared error]\label{cor:MISE.tilde.F.n.m}
Suppose that Assumptions~\ref{ass:1}, \ref{ass:3}, and \ref{ass:4} hold. Then, as $n\to \infty$, we have
\[
\begin{aligned}
\mathrm{MISE}[\widetilde{F}_{n,m}]
&= \int_0^1 \mathrm{MSE}[\widetilde{F}_{n,m}(y)] \rd y \\
&= n^{-1} \int_0^1 \sigma^2(y) \rd y - n^{-1} m^{-1/2} \int_0^1 V(y) \rd y + m^{-2} \int_0^1 B^2(y) \rd y \\
&\qquad+ \oo(n^{-1} m^{-1/2}) + \oo(m^{-2}).
\end{aligned}
\]
In particular, if $\int_0^1 B^2(y) \rd y \times \int_0^1 V(y) \rd y \neq 0$, the asymptotically optimal choice of $m$, with respect to the $\mathrm{MISE}$, is
\[
m_{\mathrm{opt}} = n^{2/3} \left[\frac{4 \int_0^1 B^2(y) \rd y}{\int_0^1 V(y) \rd y}\right]^{2/3},
\]
in which case, as $n\to \infty$,
\[
\mathrm{MISE}[\widetilde{F}_{m_{\mathrm{opt}},n}] = n^{-1} \int_0^1 \sigma^2(y) \rd y - n^{-4/3} \left[\frac{27 \, \{\int_0^1 V(y) \rd y\}^4}{256 \int_0^1 B^2(y) \rd y}\right]^{1/3} + \oo(n^{-4/3}).
\]
\end{corollary}

Next, we establish the asymptotic normality of the pseudo estimator, which is essential for constructing pointwise confidence intervals. The proof is a straightforward verification of the Lindeberg condition for double arrays. Developing uniform confidence bands for the entire curve, however, falls outside the scope of this paper and would typically require bootstrap or multiplier methods.

\begin{theorem}[Asymptotic normality]\label{thm:tilde.F.n.m.asymp.normality}
Suppose that Assumptions~\ref{ass:1}, \ref{ass:3}, and \ref{ass:4} hold, and assume that $\sigma^2(y) > 0$ in \eqref{eq:def.sigma2.V}. Then, for any $y\in (0,1)$ and $\lambda\in (0,\infty)$, we have, as $n\to \infty$,
\[
n^{1/2} \{\widetilde{F}_{n,m}(y) - F(y)\} \stackrel{d}{\longrightarrow}
\begin{cases}
\mathcal{N}(0,\sigma^2(y)), &\mbox{if } n^{1/2} m^{-1} \to 0, \\
\mathcal{N}(\lambda B(y),\sigma^2(y)), &\mbox{if } n^{1/2} m^{-1} \to \lambda,
\end{cases}
\]
where $\stackrel{d}{\longrightarrow}$ denotes the convergence in distribution (pointwise in $y$).
\end{theorem}

\subsection{Asymptotic properties of \texorpdfstring{$\smash{\widehat{F}_{n,m}}$}{hat{F}\_n,m}}\label{sec:asymp.hat.F.n.m}

We now turn to the feasible estimator $\smash{\widehat{F}_{n,m}}$, which uses estimated propensity scores. Throughout this subsection, we work under the discrete finite-support assumption on $\bb{X}$ from Section~\ref{sec:setup}. We analyze its properties by characterizing its relationship with the pseudo estimator $\smash{\widetilde{F}_{n,m}}$. We first show that the process of estimating the propensity scores introduces only a negligible amount of additional bias of order $n^{-1}$, which can be positive or negative.

\begin{proposition}[Bias]\label{prop:bias.hat.F.n.m}
Suppose that Assumptions~\ref{ass:1}--\ref{ass:3} hold. Then, for any $y\in (0,1)$, we have, as $n\to \infty$,
\[
\begin{aligned}
\Bias(\widehat{F}_{n,m}(y))
= \EE[\widehat{F}_{n,m}(y)] - F(y)
&= \Bias(\widetilde{F}_{n,m}(y)) + \OO(n^{-1}) \\
&= m^{-1} B(y) + \oo_y(m^{-1}) + \OO(n^{-1}),
\end{aligned}
\]
where $B$ is defined in \eqref{eq:def.B}.
\end{proposition}

A key finding is that the feasible estimator exhibits a notable variance reduction compared to the pseudo estimator. Proposition~\ref{prop:var.hat.F.n.m} reveals that $\smash{\widehat{F}_{n,m}}$ has a smaller asymptotic variance than $\smash{\widetilde{F}_{n,m}}$ by an explicit nonnegative correction term $n^{-1} C(y)$. This highlights a beneficial interaction between smoothing and nonparametric propensity estimation in this context. Let us remark that \citet[Eq.~(8)]{MR2569798} obtained a similar variance reduction in the KDE setting.

\begin{proposition}[Variance]\label{prop:var.hat.F.n.m}
Suppose that Assumptions~\ref{ass:1}--\ref{ass:4} hold, and define
\begin{equation}\label{eq:C.nu}
\begin{aligned}
C(y)
&= \EE\left[\frac{(1 - \pi(\bb{X}_1))}{\pi(\bb{X}_1)} \{F_{Y_1 \mid \bb{X}_1}(y)\}^2\right], \\
%%%
\nu^2(y)
&= \sigma^2(y) - C(y) \\
&= \EE\left[\frac{F_{Y_1 \mid \bb{X}_1}(y) (1 - F_{Y_1 \mid \bb{X}_1}(y))}{\pi(\bb{X}_1)}\right] + \EE\left[\{F_{Y_1 \mid \bb{X}_1}(y)\}^2\right] - \{F(y)\}^2.
\end{aligned}
\end{equation}
Then, for any $y\in (0,1)$, we have, as $n\to \infty$,
\[
\begin{aligned}
\Var(\widehat{F}_{n,m}(y))
&= \Var(\widetilde{F}_{n,m}(y)) - n^{-1} C(y) + \OO_y(n^{-1} m^{-1}) \\
&= n^{-1} \nu^2(y) - n^{-1} m^{-1/2} V(y) + \oo_y(n^{-1} m^{-1/2}),
\end{aligned}
\]
where $V(y)$ is defined in \eqref{eq:def.sigma2.V}.
\end{proposition}

\begin{remark}\label{rem:var.reduction.continuous}
The variance reduction in Proposition~\ref{prop:var.hat.F.n.m} does not rely on the auxiliary variable $\bb{X}$ being discrete with finite support. As mentioned in Section~\ref{sec:setup}, the restriction to the discrete case was adopted mainly to avoid the tedious task of treating the discrete and continuous covariate settings in parallel. The same phenomenon was proved by \citet{MR2569798} for the IPW KDE with continuous covariates. In the present CDF setting with continuous covariates, the same asymptotic results hold if the propensity scores are estimated by a Nadaraya--Watson estimator.
\end{remark}

Combining the bias and the reduced variance, we obtain the MSE for the feasible estimator. Since the leading bias term and the variance reduction term due to smoothing remain the same as for the pseudo estimator, the asymptotically optimal degree $m_{\mathrm{opt}}(y)$ is unchanged. However, the resulting MSE is smaller due to the lower asymptotic variance: typically $\nu^2(y) < \sigma^2(y)$.

\begin{corollary}[Mean squared error]\label{cor:MSE.hat.F.n.m}
Suppose that Assumptions~\ref{ass:1}--\ref{ass:4} hold. Then, for any $y\in (0,1)$, we have, as $n\to \infty$,
\[
\begin{aligned}
\mathrm{MSE}[\widehat{F}_{n,m}(y)]
&= \Var(\widehat{F}_{n,m}(y)) + \{\Bias(\widehat{F}_{n,m}(y))\}^2 \\
&= \mathrm{MSE}[\widetilde{F}_{n,m}(y)] - n^{-1} C(y) + \OO_y(n^{-1} m^{-1}) + \OO(n^{-2}) \\
&= n^{-1} \nu^2(y) - n^{-1} m^{-1/2} V(y) + m^{-2} B^2(y) + \oo_y(n^{-1} m^{-1/2}) + \oo_y(m^{-2}) + \OO(n^{-2}).
\end{aligned}
\]
In particular, if $B(y) V(y) \neq 0$, the asymptotically optimal choice of $m$, with respect to the $\mathrm{MSE}$, is
\[
m_{\mathrm{opt}}(y) = n^{2/3} \left[4 B^2(y)/V(y)\right]^{2/3},
\]
in which case, as $n\to \infty$,
\[
\mathrm{MSE}[\widehat{F}_{m_{\mathrm{opt}},n}(y)] = n^{-1} \nu^2(y) - n^{-4/3} \left[27 \, V^4(y)/\{256 B^2(y)\}\right]^{1/3} + \oo_y(n^{-4/3}).
\]
\end{corollary}

The global performance, measured by MISE, also improves for the feasible estimator. The optimal global degree $m_{\mathrm{opt}}$ remains consistent with the pseudo estimator, confirming the overall superiority of the feasible estimator in this setting.

\begin{corollary}[Mean integrated squared error]\label{cor:MISE.hat.F.n.m}
Suppose that Assumptions~\ref{ass:1}--\ref{ass:4} hold. Then, as $n\to \infty$, we have
\[
\begin{aligned}
\mathrm{MISE}[\widehat{F}_{n,m}]
&= \int_0^1 \mathrm{MSE}[\widehat{F}_{n,m}(y)] \rd y \\
&= \mathrm{MISE}[\widetilde{F}_{n,m}] - n^{-1} \int_0^1 C(y) \rd y + \OO(n^{-1} m^{-1}) + \OO(n^{-2}) \\
&= n^{-1} \int_0^1 \nu^2(y) \rd y - n^{-1} m^{-1/2} \int_0^1 V(y) \rd y + m^{-2} \int_0^1 B^2(y) \rd y \\
&\qquad+ \oo(n^{-1} m^{-1/2}) + \oo(m^{-2}) + \OO(n^{-2}).
\end{aligned}
\]
In particular, if $\int_0^1 B^2(y) \rd y \times \int_0^1 V(y) \rd y \neq 0$, the asymptotically optimal choice of $m$, with respect to the $\mathrm{MISE}$, is
\[
m_{\mathrm{opt}} = n^{2/3} \left[\frac{4 \int_0^1 B^2(y) \rd y}{\int_0^1 V(y) \rd y}\right]^{2/3},
\]
in which case, as $n\to \infty$,
\[
\mathrm{MISE}[\widehat{F}_{m_{\mathrm{opt}},n}] = n^{-1} \int_0^1 \nu^2(y) \rd y - n^{-4/3} \left[\frac{27 \, \{\int_0^1 V(y) \rd y\}^4}{256 \int_0^1 B^2(y) \rd y}\right]^{1/3} + \oo(n^{-4/3}).
\]
\end{corollary}

Finally, we establish the asymptotic normality of the feasible estimator. The conditions on the polynomial degree remain the same as for the pseudo estimator, but the asymptotic variance $\nu^2(y)$ is smaller than $\sigma^2(y)$, reflecting the efficiency gain from estimating the propensity scores. As is the case with the pseudo estimator, deriving functional limit results for this feasible process to perform uniform inference would necessitate further techniques, such as the bootstrap.

\begin{theorem}[Asymptotic normality]\label{thm:hat.F.n.m.asymp.normality}
Suppose that Assumptions~\ref{ass:1}--\ref{ass:4} hold, and assume that $\nu^2(y) > 0$ in \eqref{eq:C.nu}. Then, for any $y\in (0,1)$ and $\lambda\in (0,\infty)$, we have, as $n\to \infty$,
\[
n^{1/2} \{\widehat{F}_{n,m}(y) - F(y)\} \stackrel{d}{\longrightarrow}
\begin{cases}
\mathcal{N}(0,\nu^2(y)), &\mbox{if } n^{1/2} m^{-1} \to 0, \\
\mathcal{N}(\lambda B(y),\nu^2(y)), &\mbox{if } n^{1/2} m^{-1} \to \lambda,
\end{cases}
\]
where again the convergence in distribution is pointwise in $y$.
\end{theorem}

\section{Monte Carlo simulations}\label{sec:simulations}

In this section, we conduct a modest Monte Carlo study to evaluate three families of estimators for the CDF under MAR: the unsmoothed IPW empirical CDF, its Bernstein-smoothed version, and the \emph{integrated} IPW Gaussian KDE of \citet{MR2569798} (abbreviated I-IPW KDE). Each family is examined in two regimes: a pseudo version using known propensities and a feasible version using propensities estimated nonparametrically, as described in Section~\ref{sec:setup}. All simulations were implemented in \textsf{R}; see Appendix~\ref{app:code} for the GitHub link to the code.

\subsection{Data generating process and setup}

We consider a setting with a univariate continuous response $Y$ and a univariate discrete auxiliary covariate $X$ ($d = 1$). This entails no loss of generality for discrete covariates, given that any finite-dimensional discrete $\bb{X}$ can be recoded into a single factor by labeling the cells of the Cartesian product of its marginal categories. The data-generating process is:
\begin{itemize}\setlength\itemsep{0em}
\item $Y_i \sim \mathrm{Beta}(\alpha,\beta)$ with $(\alpha,\beta)=(0.9,0.9)$;
\item letting $F_{\mathrm{Beta}(0.9,0.9)}$ denote the $\mathrm{Beta}(0.9,0.9)$ CDF and $\Phi$ the standard normal CDF, define
\[
U_i = F_{\mathrm{Beta}(0.9,0.9)}(Y_i), \qquad
T_i = \Phi^{-1}(U_i),
\]
so that $T_i \sim \mathcal{N}(0,1)$. Next, generate
\[
X_i^\star = \rho \, T_i + \sqrt{1-\rho^{2}} \, Z_i, \qquad Z_i \sim \mathcal{N}(0,1),
\]
with $Z_i$ independent of $Y_i$, and fix $\rho = 0.6$. Finally, define the discrete auxiliary covariate by
\[
X_i = \ind_{\{X_i^\star > 0\}}.
\]
Then $X_i \sim \mathrm{Bernoulli}(p_X)$ with $p_X=0.5$. In particular, $X_i$ and $Y_i$ are dependent through the latent Gaussian construction.
\item missingness is MAR with logistic propensity
\[
\pi(x) = \PP(\delta_i = 1 \mid X_i=x) = [1+\exp\{-(\beta_0 + \beta_1 x)\}]^{-1}, \qquad x\in\{0,1\},
\]
where $(\beta_0,\beta_1)$ are chosen so that
\[
\pi(0)=0.6, \qquad \pi(1)=0.9.
\]
Equivalently,
\[
\beta_0 = \log(0.6/0.4), \qquad
\beta_1 = \log(0.9/0.1)-\log(0.6/0.4),
\]
so that $\EE[\pi(X_i)]=0.75$ and $\pi_{\min}=0.6$.
\end{itemize}
For the feasible estimators, the propensity is estimated nonparametrically using
\[
\widehat{\pi}(x)=\frac{\sum_{i=1}^n \delta_i \ind_{\{X_i = x\}}}{\sum_{i=1}^n \ind_{\{X_i = x\}}}.
\]

We compare six estimators of the CDF $F$, based on the true IPW weights $W_i=\delta_i/\pi(X_i)$ in the pseudo setting and the estimated IPW weights $\widehat{W}_i=\delta_i/\widehat{\pi}(X_i)$ in the feasible setting:
\begin{enumerate}[label=(\roman*)]\setlength\itemsep{0em}
\item the unsmoothed pseudo estimator $\widetilde F_n$ using $W_i$;
\item the pseudo I-IPW KDE, obtained by integrating the IPW KDE of \citet{MR2569798} constructed with $W_i$;
\item the Bernstein-smoothed pseudo estimator $\widetilde F_{n,m}=\mathcal B_m(\widetilde F_n)$ (smoothing $\widetilde F_n$);
\item the unsmoothed feasible estimator $\widehat{F}_n$ using $\widehat{W}_i$;
\item the feasible I-IPW KDE, obtained by integrating the IPW KDE built with $\widehat{W}_i$;
\item the Bernstein-smoothed feasible estimator $\widehat{F}_{n,m}=\mathcal B_m(\widehat{F}_n)$ (smoothing $\widehat{F}_n$).
\end{enumerate}

\paragraph{Simulation parameters}
We conduct two Monte Carlo experiments, each with $1000$ replications. The first (Section~\ref{sec:sample.size.experiment}) examines the impact of the sample size under the benchmark MAR mechanism defined above, with
\[
n \in \{25,50,100,200,400,800,1600,3200,6400\}.
\]
The second (Section~\ref{sec:missing.rate.experiment}) examines the impact of the missing rate at the fixed sample size $n=400$. In that experiment, we retain the same latent Gaussian dependence between $X$ and $Y$, and we vary only the intercept of the logistic MAR mechanism, keeping the slope coefficient $\beta_1$ fixed at its benchmark value. More precisely, for each target missing rate $r\in\{5\%,10\%,15\%,20\%,25\%,30\%,35\%,40\%\}$, we define
\[
\pi_r(x)=\{1+\exp[-(\beta_0(r)+\beta_1 x)]\}^{-1}, \qquad x\in\{0,1\},
\]
where $\beta_0(r)$ is chosen so that
\[
1-\EE[\pi_r(X_i)] = r.
\]

\paragraph{Performance measures}
For each replication, we compute the integrated squared error,
\[
\mathrm{ISE}(\widehat{F})=\int_0^1 \{\widehat{F}(y)-F(y)\}^2 \, \rd y,
\]
and the boundary ISE,
\[
\mathrm{BISE}(\widehat{F}) = \frac{1}{2\delta} \int_{[0,\delta] \cup [1-\delta,1]} \{\widehat{F}(y)-F(y)\}^2 \, \rd y, \qquad \delta = n^{-2/3}.
\]
For each estimator, we compute the integrals numerically using the cubature routine \texttt{adaptIntegrate}. We summarize the distribution of the ISEs and BISEs across replications by their mean and standard deviation.

\subsection{Smoothing-parameter selection via optimized LSCV}

For the smoothed Bernstein estimator $\smash{\widehat{F}_{n,m}}$, the polynomial degree $m$ is chosen by least-squares cross-validation (LSCV). The derivations below are completely analogous for $\smash{\widetilde{F}_{n,m}}$, and thus are omitted. The LSCV criterion (up to an additive constant independent of $m$) is
\begin{equation}\label{eq:LSCV-final}
\mathrm{LSCV}(m) =
\underbrace{\int_0^1 \widehat{F}_{n,m}(y)^2 \, \rd y}_{\text{Term 1}}
-
\underbrace{\frac{2}{n}\sum_{i=1}^n \widehat{W}_i \int_{Y_i}^{1} \widehat{F}_{n,m}^{(-i)}(y) \, \rd y}_{\text{Term 2}},
\end{equation}
where $\smash{\widehat{F}_{n,m}^{(-i)}}$ denotes the leave-one-out version (i.e., $\smash{\widehat{F}_{n,m}}$ with the $i$th data point left out).

\paragraph{Term 1}
Using bilinearity and the Bernstein basis, we obtain
\[
\int_0^1 \widehat{F}_{n,m}(y)^2 \, \rd y
= \sum_{k=0}^m \sum_{\ell=0}^m \widehat{F}_n(k/m) \widehat{F}_n(\ell/m) \binom{m}{k} \binom{m}{\ell} B(k+\ell+1, 2m-k-\ell+1),
\]
where $B(\cdot,\cdot)$ denotes the Beta function \citep[see, e.g.,][Section~5.12]{NIST2010}. Term 1 is evaluated in $\OO(m^2)$ operations.

\paragraph{Term 2}
By Fubini's theorem and the identity $F(y) = \EE[\ind_{\{y \geq Y\}}]$,
\[
\int_0^1 \widehat{F}_{n,m}(y) F(y) \, \rd y
= \EE\left[\int_Y^1 \widehat{F}_{n,m}(y) \, \rd y\right],
\]
which motivates the leave-one-out estimator in \eqref{eq:LSCV-final}. Expanding the Bernstein operator and integrating the Bernstein basis functions gives
\[
\int_{Y_i}^{1} b_{m,k}(y) \, \rd y = \frac{1-\mathrm{pbeta}(Y_i; k+1, m-k+1)}{m+1},
\]
where $\mathrm{pbeta}(\cdot \, ; \alpha, \beta)$ denotes the CDF of the $\mathrm{Beta}(\alpha,\beta)$ distribution, so that
\[
\int_{Y_i}^{1} \widehat{F}_{n,m}^{(-i)}(y) \, \rd y
= \sum_{k=0}^m \widehat{F}_n^{(-i)}(k/m) \frac{1-\mathrm{pbeta}(Y_i; k+1, m-k+1)}{m+1}.
\]
Here, $\widehat{F}_n^{(-i)}(k/m)$ can be obtained without recomputing from scratch via
\[
\widehat{F}_n^{(-i)}(k/m) = \frac{n \widehat{F}_n(k/m) - \widehat{W}_i \ind_{\{Y_i \leq k/m\}}}{n-1}.
\]
Thus Term~2 is computed in $\OO(nm)$ operations. Overall, $\mathrm{LSCV}(m)$ is evaluated in $\OO(m^2+nm)$ operations for each $m$. We search $m$ over a data-dependent grid
\[
m \in \{m_{\min},\ldots,m_{\max}(n)\}, \qquad
m_{\max}(n)=\min(\lfloor c n^{2/3}\rfloor,~m_{\mathrm{cap}},~n),
\]
with $(m_{\min},m_{\mathrm{cap}},c)=(1,300,5)$ in our \textsf{R} script. The selected $m^\star$ minimizes \eqref{eq:LSCV-final}.

For the I-IPW KDE competitor, the kernel bandwidth $h$ is chosen by an entirely analogous LSCV scheme; see Section~2.2 of \citet{MR2569798} for details.

\subsection{Impact of the sample size}\label{sec:sample.size.experiment}

This first experiment examines how estimator performance evolves with the sample size $n$ under the benchmark MAR mechanism described above. Tables~\ref{tab:pseudo_ISE_sample_size} and~\ref{tab:pseudo_BISE_sample_size} report, respectively, the ISE and BISE summaries for the pseudo estimators. Tables~\ref{tab:feasible_ISE_sample_size} and~\ref{tab:feasible_BISE_sample_size} report the analogous quantities for the feasible estimators. Each table reports the mean and standard deviation across replications, with columns ordered as the IPW empirical estimator (labeled Unsmoothed), the I-IPW KDE, and the Bernstein estimator. \bigskip

% Sample-size study: pseudo estimators, ISE
\begin{table}[H]
\centering
\small
\begingroup
\renewcommand{\arraystretch}{       1.15}
\begin{tabular}{|r|ccc|ccc|}
\hline
$n$ & \multicolumn{3}{|c|}{Mean ISE ($\times 10^{8}$)} & \multicolumn{3}{|c|}{Standard deviation ISE ($\times 10^{8}$)} \\
\hline
  & Unsmoothed & I-IPW KDE & Bernstein & Unsmoothed & I-IPW KDE & Bernstein \\
\hline
25 & 1647074 & 1358550 & \textbf{879985} & 1608140 & 1426764 & \textbf{1309825} \\[-0.5mm]
50 & 815268 & 699262 & \textbf{443236} & 784977 & 725813 & \textbf{599672} \\[-0.5mm]
100 & 391641 & 340189 & \textbf{224728} & 427206 & 399121 & \textbf{341277} \\[-0.5mm]
200 & 199960 & 179626 & \textbf{114586} & 200241 & 190795 & \textbf{150616} \\[-0.5mm]
400 & 97876 & 90737 & \textbf{64225} & 99395 & 97440 & \textbf{79036} \\[-0.5mm]
800 & 48878 & 46379 & \textbf{38236} & 52316 & 52045 & \textbf{42933} \\[-0.5mm]
1600 & 24831 & \textbf{23921} & 24491 & 25416 & 25231 & \textbf{22532} \\[-0.5mm]
3200 & 12384 & \textbf{12087} & 12974 & 13210 & 13174 & \textbf{12594} \\[-0.5mm]
6400 & 6707 & 6633 & \textbf{6395} & 6742 & 6738 & \textbf{6702} \\
\hline
\end{tabular}
\endgroup
\caption{ISE statistics for pseudo estimators as a function of the sample size.}
\label{tab:pseudo_ISE_sample_size}
\end{table}

% Sample-size study: pseudo estimators, BISE
\begin{table}[H]
\centering
\small
\begingroup
\renewcommand{\arraystretch}{       1.15}
\begin{tabular}{|r|ccc|ccc|}
\hline
$n$ & \multicolumn{3}{|c|}{Mean BISE ($\times 10^{8}$)} & \multicolumn{3}{|c|}{Standard deviation BISE ($\times 10^{8}$)} \\
\hline
  & Unsmoothed & I-IPW KDE & Bernstein & Unsmoothed & I-IPW KDE & Bernstein \\
\hline
25 & 1139302 & 1667561 & \textbf{782580} & 1279327 & 1685791 & \textbf{1091595} \\[-0.5mm]
50 & 517937 & 894690 & \textbf{406378} & 596773 & 876662 & \textbf{537461} \\[-0.5mm]
100 & 229159 & 444997 & \textbf{195663} & 317125 & 441765 & \textbf{290511} \\[-0.5mm]
200 & 104407 & 231867 & \textbf{94370} & 133667 & 210253 & \textbf{128030} \\[-0.5mm]
400 & 54455 & 125259 & \textbf{51809} & 72594 & 116153 & \textbf{71233} \\[-0.5mm]
800 & 25206 & 63876 & \textbf{24492} & 35856 & 60658 & \textbf{34790} \\[-0.5mm]
1600 & \textbf{11926} & 28213 & 12104 & \textbf{15886} & 25092 & 16289 \\[-0.5mm]
3200 & 6052 & 14266 & \textbf{6047} & 9102 & 13351 & \textbf{9078} \\[-0.5mm]
6400 & \textbf{3049} & 5101 & 3078 & \textbf{4236} & 5406 & 4278 \\
\hline
\end{tabular}
\endgroup
\caption{Boundary ISE statistics for pseudo estimators as a function of the sample size.}
\label{tab:pseudo_BISE_sample_size}
\end{table}

% Sample-size study: feasible estimators, ISE
\begin{table}[H]
\centering
\small
\begingroup
\renewcommand{\arraystretch}{       1.15}
\begin{tabular}{|r|ccc|ccc|}
\hline
$n$ & \multicolumn{3}{|c|}{Mean ISE ($\times 10^{8}$)} & \multicolumn{3}{|c|}{Standard deviation ISE ($\times 10^{8}$)} \\
\hline
  & Unsmoothed & I-IPW KDE & Bernstein & Unsmoothed & I-IPW KDE & Bernstein \\
\hline
25 & 954524 & 713278 & \textbf{239168} & 798381 & 589142 & \textbf{550044} \\[-0.5mm]
50 & 455141 & 346095 & \textbf{120758} & 401704 & 325848 & \textbf{281482} \\[-0.5mm]
100 & 220940 & 174579 & \textbf{67042} & 200497 & 172123 & \textbf{143507} \\[-0.5mm]
200 & 112125 & 93299 & \textbf{40544} & 102896 & 93909 & \textbf{69979} \\[-0.5mm]
400 & 56231 & 49072 & \textbf{27131} & 53764 & 51132 & \textbf{37264} \\[-0.5mm]
800 & 26230 & 23670 & \textbf{18487} & 21966 & 21336 & \textbf{14154} \\[-0.5mm]
1600 & 13543 & \textbf{12685} & 13695 & 11894 & 11753 & \textbf{8238} \\[-0.5mm]
3200 & 6939 & \textbf{6648} & 7774 & 6027 & 5993 & \textbf{5908} \\[-0.5mm]
6400 & 3794 & 3722 & \textbf{3483} & 3371 & \textbf{3367} & 3502 \\
\hline
\end{tabular}
\endgroup
\caption{ISE statistics for feasible estimators as a function of the sample size.}
\label{tab:feasible_ISE_sample_size}
\end{table}

% Sample-size study: feasible estimators, BISE
\begin{table}[H]
\centering
\small
\begingroup
\renewcommand{\arraystretch}{       1.15}
\begin{tabular}{|r|ccc|ccc|}
\hline
$n$ & \multicolumn{3}{|c|}{Mean BISE ($\times 10^{8}$)} & \multicolumn{3}{|c|}{Standard deviation BISE ($\times 10^{8}$)} \\
\hline
  & Unsmoothed & I-IPW KDE & Bernstein & Unsmoothed & I-IPW KDE & Bernstein \\
\hline
25 & 363853 & 1009706 & \textbf{46069} & 357238 & 879450 & \textbf{91124} \\[-0.5mm]
50 & 118404 & 500208 & \textbf{16164} & 120360 & 439568 & \textbf{28589} \\[-0.5mm]
100 & 38755 & 265221 & \textbf{6693} & 36463 & 194615 & \textbf{3494} \\[-0.5mm]
200 & 12347 & 141440 & \textbf{3333} & 11561 & 88932 & \textbf{706} \\[-0.5mm]
400 & 3997 & 73016 & \textbf{1682} & 3858 & 38645 & \textbf{319} \\[-0.5mm]
800 & 1099 & 37007 & \textbf{797} & 1080 & 16500 & \textbf{215} \\[-0.5mm]
1600 & \textbf{289} & 17649 & 311 & 296 & 7321 & \textbf{121} \\[-0.5mm]
3200 & \textbf{57} & 8144 & 88 & 61 & 2052 & \textbf{48} \\[-0.5mm]
6400 & \textbf{12} & 2213 & 21 & 12 & 1271 & \textbf{12} \\
\hline
\end{tabular}
\endgroup
\caption{Boundary ISE statistics for feasible estimators as a function of the sample size.}
\label{tab:feasible_BISE_sample_size}
\end{table}

Across Tables~\ref{tab:pseudo_ISE_sample_size}--\ref{tab:feasible_BISE_sample_size}, both the mean and the standard deviation of the error measures decrease steadily with $n$, as expected. For global accuracy, measured by the ISE, the Bernstein estimator clearly dominates in the small-to-medium sample-size range in both the pseudo and feasible regimes: from $n=25$ up to $n=800$, it has the smallest mean ISE and the smallest standard deviation in every reported case. For larger sample sizes, the three procedures become much closer, which is exactly what one would expect since they share the same first-order asymptotics. In particular, in both the pseudo and feasible regimes, the I-IPW KDE has a slightly smaller \emph{mean} ISE at $n=1600$ and $n=3200$, while the Bernstein estimator is again best at $n=6400$; these differences are numerically very small compared with the pronounced gains seen at smaller $n$. The standard deviations remain smallest for the Bernstein estimator throughout essentially the whole range, up to negligible near-ties at the very largest sample sizes.

For boundary accuracy, the picture is even sharper. The I-IPW KDE is consistently disfavored by the BISE criterion, often by a wide margin, reflecting the familiar boundary spillover of the Gaussian kernel. This boundary effect becomes less damaging as $n$ grows, because the selected bandwidth decreases and the boundary region $[0,\delta]\cup[1-\delta,1]$ itself shrinks. Even so, the BISE normalization by $2\delta$ keeps boundary leakage visible, so the KDE remains clearly inferior on that metric. By contrast, the Bernstein estimator is best in the small-to-medium range for both pseudo and feasible estimators, in mean as well as in standard deviation. For large sample sizes, the unsmoothed IPW estimator and the Bernstein estimator become nearly indistinguishable for BISE, with the unsmoothed estimator occasionally having a marginally smaller mean in the pseudo regime ($n=1600$ and $n=6400$) and more systematically in the feasible regime from $n=1600$ onward, whereas Bernstein typically retains the smaller or comparable dispersion.

From a practical standpoint, the small and medium sample sizes are the most relevant, and there the Bernstein estimator wins decisively, both globally and near the boundary.

\subsection{Impact of the missing rate}\label{sec:missing.rate.experiment}

This second experiment examines how estimator performance changes as the missing rate increases. We fix $n=400$ and retain the same latent Gaussian dependence structure between $X$ and $Y$, while varying only the intercept of the logistic MAR mechanism to obtain target missing rates $5\%$, $10\%$, $15\%$, $20\%$, $25\%$, $30\%$, $35\%$, and $40\%$. Tables~\ref{tab:pseudo_ISE_missing_rate} and~\ref{tab:pseudo_BISE_missing_rate} report, respectively, the ISE and BISE summaries for the pseudo estimators. Tables~\ref{tab:feasible_ISE_missing_rate} and~\ref{tab:feasible_BISE_missing_rate} report the analogous quantities for the feasible estimators. Each table reports the mean and standard deviation across replications, with columns ordered as the IPW empirical estimator (labeled Unsmoothed), the I-IPW KDE, and the Bernstein estimator. \bigskip

% Missing-rate study: pseudo estimators, ISE
\begin{table}[H]
\centering
\small
\begingroup
\renewcommand{\arraystretch}{       1.15}
\begin{tabular}{|r|ccc|ccc|}
\hline
Missing rate (\%) & \multicolumn{3}{|c|}{Mean ISE ($\times 10^{8}$)} & \multicolumn{3}{|c|}{Standard deviation ISE ($\times 10^{8}$)} \\
\hline
  & Unsmoothed & I-IPW KDE & Bernstein & Unsmoothed & I-IPW KDE & Bernstein \\
\hline
5 & 49846 & 45040 & \textbf{30627} & 45293 & 43778 & \textbf{33023} \\[-0.5mm]
10 & 59864 & 54385 & \textbf{37470} & 54372 & 52865 & \textbf{40325} \\[-0.5mm]
15 & 71677 & 65463 & \textbf{47369} & 70809 & 68961 & \textbf{56305} \\[-0.5mm]
20 & 79770 & 73324 & \textbf{51737} & 79509 & 77866 & \textbf{63512} \\[-0.5mm]
25 & 101121 & 93653 & \textbf{67903} & 105151 & 102533 & \textbf{87499} \\[-0.5mm]
30 & 125709 & 117153 & \textbf{82657} & 131487 & 128308 & \textbf{103824} \\[-0.5mm]
35 & 146234 & 136613 & \textbf{99886} & 165434 & 161487 & \textbf{136999} \\[-0.5mm]
40 & 176410 & 163611 & \textbf{121891} & 193369 & 188011 & \textbf{159135} \\
\hline
\end{tabular}
\endgroup
\caption{ISE statistics for pseudo estimators as a function of the missing rate.}
\label{tab:pseudo_ISE_missing_rate}
\end{table}

% Missing-rate study: pseudo estimators, BISE
\begin{table}[H]
\centering
\small
\begingroup
\renewcommand{\arraystretch}{       1.15}
\begin{tabular}{|r|ccc|ccc|}
\hline
Missing rate (\%) & \multicolumn{3}{|c|}{Mean BISE ($\times 10^{8}$)} & \multicolumn{3}{|c|}{Standard deviation BISE ($\times 10^{8}$)} \\
\hline
  & Unsmoothed & I-IPW KDE & Bernstein & Unsmoothed & I-IPW KDE & Bernstein \\
\hline
5 & 10047 & 57672 & \textbf{8202} & 11454 & 36879 & \textbf{9716} \\[-0.5mm]
10 & 16493 & 65693 & \textbf{14816} & 19983 & 50029 & \textbf{18638} \\[-0.5mm]
15 & 26482 & 79045 & \textbf{25054} & 36023 & 63341 & \textbf{35503} \\[-0.5mm]
20 & 36640 & 98535 & \textbf{35216} & 51437 & 88029 & \textbf{50568} \\[-0.5mm]
25 & 52222 & 116896 & \textbf{49833} & 69333 & 109618 & \textbf{68119} \\[-0.5mm]
30 & 68140 & 148497 & \textbf{64454} & 87625 & 134854 & \textbf{85735} \\[-0.5mm]
35 & 89121 & 182761 & \textbf{85419} & 122119 & 169637 & \textbf{120135} \\[-0.5mm]
40 & 115383 & 209867 & \textbf{112690} & 154893 & 206946 & \textbf{154663} \\
\hline
\end{tabular}
\endgroup
\caption{Boundary ISE statistics for pseudo estimators as a function of the missing rate.}
\label{tab:pseudo_BISE_missing_rate}
\end{table}

% Missing-rate study: feasible estimators, ISE
\begin{table}[H]
\centering
\small
\begingroup
\renewcommand{\arraystretch}{       1.15}
\begin{tabular}{|r|ccc|ccc|}
\hline
Missing rate (\%) & \multicolumn{3}{|c|}{Mean ISE ($\times 10^{8}$)} & \multicolumn{3}{|c|}{Standard deviation ISE ($\times 10^{8}$)} \\
\hline
  & Unsmoothed & I-IPW KDE & Bernstein & Unsmoothed & I-IPW KDE & Bernstein \\
\hline
5 & 44112 & 39257 & \textbf{25455} & 39272 & 37832 & \textbf{27858} \\[-0.5mm]
10 & 48057 & 42640 & \textbf{26531} & 41309 & 39581 & \textbf{28731} \\[-0.5mm]
15 & 51462 & 45414 & \textbf{27984} & 46962 & 44899 & \textbf{33173} \\[-0.5mm]
20 & 52249 & 45862 & \textbf{26533} & 46526 & 44178 & \textbf{32603} \\[-0.5mm]
25 & 56802 & 49492 & \textbf{27196} & 47827 & 45563 & \textbf{33449} \\[-0.5mm]
30 & 62124 & 53937 & \textbf{28045} & 56075 & 52849 & \textbf{38080} \\[-0.5mm]
35 & 69724 & 59972 & \textbf{30044} & 61193 & 57088 & \textbf{41029} \\[-0.5mm]
40 & 75947 & 64513 & \textbf{30308} & 61937 & 57043 & \textbf{42423} \\
\hline
\end{tabular}
\endgroup
\caption{ISE statistics for feasible estimators as a function of the missing rate.}
\label{tab:feasible_ISE_missing_rate}
\end{table}

% Missing-rate study: feasible estimators, BISE
\begin{table}[H]
\centering
\small
\begingroup
\renewcommand{\arraystretch}{       1.15}
\begin{tabular}{|r|ccc|ccc|}
\hline
Missing rate (\%) & \multicolumn{3}{|c|}{Mean BISE ($\times 10^{8}$)} & \multicolumn{3}{|c|}{Standard deviation BISE ($\times 10^{8}$)} \\
\hline
  & Unsmoothed & I-IPW KDE & Bernstein & Unsmoothed & I-IPW KDE & Bernstein \\
\hline
5 & 2970 & 49386 & \textbf{1624} & 2803 & 24490 & \textbf{434} \\[-0.5mm]
10 & 3126 & 52148 & \textbf{1616} & 3014 & 27510 & \textbf{412} \\[-0.5mm]
15 & 3387 & 57372 & \textbf{1629} & 3273 & 29030 & \textbf{438} \\[-0.5mm]
20 & 3640 & 66451 & \textbf{1661} & 3650 & 36393 & \textbf{357} \\[-0.5mm]
25 & 3827 & 69887 & \textbf{1673} & 3608 & 37206 & \textbf{336} \\[-0.5mm]
30 & 4187 & 82539 & \textbf{1700} & 4416 & 42584 & \textbf{346} \\[-0.5mm]
35 & 4550 & 94309 & \textbf{1720} & 4680 & 48375 & \textbf{381} \\[-0.5mm]
40 & 5079 & 106013 & \textbf{1745} & 5100 & 55943 & \textbf{325} \\
\hline
\end{tabular}
\endgroup
\caption{Boundary ISE statistics for feasible estimators as a function of the missing rate.}
\label{tab:feasible_BISE_missing_rate}
\end{table}

Tables~\ref{tab:pseudo_ISE_missing_rate}--\ref{tab:feasible_BISE_missing_rate} show that increasing missingness makes the problem harder, as expected. The mean and standard deviation of the ISE rise with the missing rate for all three methods in both regimes, and the same overall deterioration is visible for BISE as well, especially for the unsmoothed estimator and the I-IPW KDE. The Bernstein estimator is uniformly best across the board: in both the pseudo and feasible regimes, it has the smallest mean and the smallest standard deviation for both ISE and BISE at every reported missing rate. This is fully consistent with the sample-size experiment, since the present design fixes $n=400$, which lies in the range where Bernstein smoothing already showed its clearest advantage.

The boundary-spillover effect of the I-IPW KDE is again very noticeable. For ISE, the KDE typically sits between the unsmoothed and Bernstein estimators, but for BISE it is systematically much worse than the other two procedures, often by a very large margin. The feasible Bernstein estimator is particularly strong: even at $40\%$ missingness, it still has the smallest mean ISE and BISE and the smallest dispersion by a comfortable margin. More generally, while higher missingness degrades performance for all methods, the Bernstein smoother remains the most robust procedure in this experiment, both globally and at the boundary.

\subsection{Heuristic order of the LSCV-selected degree}\label{sec:LSCV.bandwidth.asymptotics}

A full proof of the asymptotic behavior of the LSCV selector is beyond the scope of this paper, but the expected scale of $m^\star$ is quite clear and is the exact Bernstein analogue of the kernel-CDF argument of \citet{MR1666695}. To keep the notation light, write the oracle version of \eqref{eq:LSCV-final} as
\[
H_n(m) \leqdef \int_0^1 \widetilde{F}_{n,m}(y)^2 \, \rd y - \frac{2}{n} \sum_{i=1}^n W_i \int_{Y_i}^1 \widetilde{F}_{n,m}^{(-i)}(y) \, \rd y, \qquad W_i = \delta_i/\pi(\bb{X}_i),
\]
and set $R_n(m) \leqdef \mathrm{MISE}[\widetilde{F}_{n,m}]$. Since $\EE[W_i \ind_{\{Y_i \leq y\}}] = F(y)$, since $\widetilde{F}_{n,m}^{(-i)}$ is independent of $(W_i,Y_i)$, and since $\EE[\widetilde{F}_{n,m}^{(-i)}(y)] = \mathcal{B}_m(F)(y)$, we get
\[
\begin{aligned}
\EE[H_n(m)]
&= \int_0^1 \EE[\widetilde{F}_{n,m}(y)^2] \rd y - 2 \int_0^1 F(y) \mathcal{B}_m(F)(y) \rd y \\
&= \int_0^1 \Var(\widetilde{F}_{n,m}(y)) \rd y + \int_0^1 \{\mathcal{B}_m(F)(y) - F(y)\}^2 \rd y - \int_0^1 F(y)^2 \rd y \\
&= R_n(m) - \int_0^1 F(y)^2 \rd y.
\end{aligned}
\]
Thus, up to the irrelevant constant $-\int_0^1 F(y)^2 \rd y$ (which does not depend on $m$), the oracle LSCV curve is an unbiased estimator of the oracle MISE curve.

The key binomial moment identities
\[
\sum_{k=0}^m (k/m - y) \, b_{m,k}(y) = 0, \qquad \sum_{k=0}^m (k/m - y)^2 \, b_{m,k}(y) = \frac{y(1-y)}{m}
\]
show that the binomial basis behaves like a kernel with effective bandwidth of order $m^{-1/2}$. For this reason, one expects the proof of \citet{MR1666695} to carry over after the substitution $h = m^{-1/2}$. More precisely, one should expand $H_n(m) - \EE[H_n(m)]$ as a linear term, a degenerate quadratic term, and a diagonal term, exactly as in their proof. If
\[
J_n \leqdef \int_0^1 \left( \{\widetilde{F}_n(y) - F(y)\}^2 - \EE[\{\widetilde{F}_n(y) - F(y)\}^2] \right) \rd y,
\]
then the Bernstein analogue of their Theorem~1 should give
\[
H_n(m) + J_n = R_n(m) + \Delta_n(m),
\]
where, for every $\delta > 0$,
\[
\Delta_n(m) = \OO\left(n^{\delta} \left\{ n^{-1/2} m^{-3/2} + n^{-1} m^{-3/4} + n^{-3/2} \right\}\right)
\]
with probability tending to $1$, uniformly for $m = n^{2/3} u$ with $u$ in any fixed compact subset of $(0,\infty)$. These are exactly the Bernstein counterparts of the terms $n^{-1/2} h^3$, $n^{-1} h^{3/2}$, and $n^{-3/2}$ in Equation~(4) of \citet{MR1666695}. On the scale $m = n^{2/3} u$, the right-hand side is $o(n^{-4/3})$ as soon as $\delta < 1/6$.

Now let
\[
A_V \leqdef \int_0^1 V(y) \rd y, \qquad A_B \leqdef \int_0^1 B^2(y) \rd y,
\]
and note from Corollary~\ref{cor:MISE.tilde.F.n.m} that
\[
R_n(m) = n^{-1} \int_0^1 \sigma^2(y) \rd y - n^{-1} m^{-1/2} A_V + m^{-2} A_B + \oo(n^{-1} m^{-1/2}) + \oo(m^{-2}).
\]
If $m = n^{\alpha}$, then
\[
R_n(m) - n^{-1} \int_0^1 \sigma^2(y) \rd y = A_B n^{-2\alpha} - A_V n^{-1-\alpha/2} + o\left(n^{-2\alpha} + n^{-1-\alpha/2}\right).
\]
Hence $\alpha < 2/3$ makes the positive bias term dominate, while $\alpha > 2/3$ makes the variance gain smaller than the $n^{-4/3}$ improvement available on the balanced scale. Thus the only possible order is $\alpha = 2/3$. Writing $m = n^{2/3} u$, we get
\[
R_n(m) = n^{-1} \int_0^1 \sigma^2(y) \rd y + n^{-4/3} \{\psi(u) + o(1)\}, \qquad \psi(u) = A_B u^{-2} - A_V u^{-1/2},
\]
and
\[
\psi'(u) = -2 A_B u^{-3} + \frac{1}{2} A_V u^{-3/2}.
\]
Therefore $\psi$ has the unique minimizer
\[
u_0 = \left(\frac{4 A_B}{A_V}\right)^{2/3} = \left[\frac{4 \int_0^1 B^2(y) \rd y}{\int_0^1 V(y) \rd y}\right]^{2/3}.
\]
Fix $\varepsilon > 0$. Since $\psi$ has a unique minimum at $u_0$, there exists $c_{\varepsilon} > 0$ such that $\psi(u) \geq \psi(u_0) + 3 c_{\varepsilon}$ whenever $|u-u_0| \geq \varepsilon$ and $u$ stays in a fixed compact set containing $u_0$. The uniform bound on $\Delta_n(m)$ is $o(n^{-4/3})$, so the same inequality holds for $H_n(n^{2/3} u)$ with probability tending to $1$. This gives the heuristic conclusion
\[
m^\star n^{-2/3} \to u_0,
\]
in probability as $n\to\infty$, and in particular
\[
\PP\left(c_1 n^{2/3} \leq m^\star \leq c_2 n^{2/3}\right) \to 1,
\]
for any $0 < c_1 < u_0 < c_2 < \infty$. For the feasible selector, Corollary~\ref{cor:MISE.hat.F.n.m} shows that propensity estimation only changes the $m$-free $n^{-1}$ term and adds an $m$-dependent remainder $\OO(n^{-1} m^{-1})$, which is $o(n^{-4/3})$ when $m$ is of order $n^{2/3}$, so the same asymptotic scale should prevail. Thus the cap $m_{\max}(n) \propto n^{2/3}$ used in Section~4.2 is not ad hoc.

Finally, this heuristic is also in line with what is already known for Bernstein polynomial smoothing in the full-data CDF setting. The expansions in \citet{MR2960952} show that the leading bias and variance terms balance when $m$ is of order $n^{2/3}$, which plays the same heuristic role for Bernstein smoothing that the bandwidth order $n^{-1/3}$ plays for classical kernel CDF smoothing. See also \citet{MR2925964} for boundary-region properties. We are not aware of a published result establishing the analogue of $m^\star n^{-2/3} \to u_0$ for an LSCV selector in the Bernstein CDF setting. Nevertheless, \citet{MR3473628} develops data-driven Bernstein estimators with random degree by minimizing estimated pointwise MSE or global MISE and proves consistency and convergence-rate results, with the search for $m$ restricted to intervals proportional to $n^{2/3}$.

Several closely related Bernstein-based results are also worth noting, although none of them proves an asymptotic theorem for an LSCV-selected degree of a Bernstein CDF estimator. In the multivariate bounded-support setting, \citet{MR3916632} propose a Bernstein polynomial model for multivariate distribution and density estimation, choose the coordinatewise degrees by a change-point rule, and obtain a nearly parametric rate in mean $\chi^2$-divergence for the resulting density estimator. For univariate densities on $[0,1]$ that may be unbounded at $0$, \citet{MR2351744} establish uniform weak and strong consistency on compact subsets of $(0,1)$ and develop both least-squares and likelihood cross-validation rules for selecting the smoothing parameter. For copula densities that may be unbounded at the corners, \citet{MR3143795} establish boundary and interior asymptotic properties of the Bernstein estimator and study a least-squares cross-validation rule for the smoothing parameter, with simulations examining finite-sample performance.

\section{Real-data application}\label{sec:application}

We highlight the method's utility in the US National Health and Nutrition Examination Survey (NHANES) 2017--2018 cycle \citep{NCHS_DEMO_J_2017_2018}; see
\citet{doi:10.32614/CRAN.package.nhanesA} for retrieving the dataset. The continuous outcome is fasting plasma glucose (\texttt{LBXGLU}, mg/dL). Our analysis is restricted throughout to the fasting laboratory subsample, namely the individuals whose records appear in the NHANES glucose file \texttt{GLU\_J}. Hence the survey design that determines eligibility for fasting-lab measurement is conditioned upon, rather than treated as part of the missing-data mechanism. Within this subsample, let $G_i=\texttt{LBXGLU}_i$ denote raw fasting plasma glucose and let
\[
\delta_i = \ind_{\{G_i\ \text{observed}\}}.
\]
Thus the missingness considered here is item nonresponse for fasting plasma glucose within the fasting-lab subsample. Conditioning on the fully observed auxiliary variable $X_i$, we impose the working MAR assumption
\[
\PP(\delta_i = 1 \mid G_i, X_i)
= \PP(\delta_i = 1 \mid X_i)
= \pi(X_i).
\]
The propensities are unknown, so we use the feasible estimators $\smash{\widehat{F}_n}$ and $\smash{\widehat{F}_{n,m}}$ defined in Section~\ref{sec:setup}.

The auxiliary variable $X$ is a discrete 4-level factor formed as the cross of the NHANES exam period indicator \texttt{RIDEXMON} and sex \texttt{RIAGENDR}. In NHANES coding, \texttt{RIDEXMON}~$\in\{1,2\}$ denotes examination period (1: November--April, 2: May--October) and \texttt{RIAGENDR}~$\in\{1,2\}$ denotes sex (1: Male, 2: Female). We map the four cells to
\[
\begin{aligned}
&X=1: \text{November--April, Male}, \\
&X=2: \text{November--April, Female}, \\
&X=3: \text{May--October, Male}, \\
&X=4: \text{May--October, Female}.
\end{aligned}
\]
Accordingly,
\[
\pi(X_i)=
\begin{cases}
\pi_1, & X_i=1,\\
\pi_2, & X_i=2,\\
\pi_3, & X_i=3,\\
\pi_4, & X_i=4,
\end{cases}
\]
so the observation probability is constant within each sex-by-exam-period cell and may vary across cells. In the feasible estimator, it is estimated by the corresponding cellwise observed fraction
\[
\widehat{\pi}(x)=\frac{\sum_{i=1}^n \delta_i \ind_{\{X_i = x\}}}{\sum_{i=1}^n \ind_{\{X_i = x\}}},
\qquad x\in \{1,2,3,4\}.
\]

As explained in Remark~\ref{rem:rescaling}, the Bernstein operator $\mathcal B_m$ is built from the binomial kernel $b_{m,k}(y)$, whose argument lies in $[0,1]$. Thus the method is not limited to intrinsically $[0,1]$-valued responses; for fasting plasma glucose, we therefore map the raw outcome monotonically to $[0,1]$ using
\begin{equation}\label{eq:mapping}
Y_i = \min\left\{\, \max\left\{ \frac{G_i - a}{\,b - a\,},\, 0 \right\},\, 1 \right\},
\qquad (a,b)=(40,460)\ \text{mg/dL}.
\end{equation}
As in Section~\ref{sec:simulations}, the Bernstein degree $m$ is selected by leave-one-out LSCV.

Table~\ref{tab:realdata.cell.propensities} reports the four cells together with their cell sizes and estimated observation probabilities $\widehat{\pi}(x)$. Table~\ref{tab:realdata.feasible} reports the effective sample size after requiring $X$ to be observed, the overall observed fraction $n^{-1} \sum_{i=1}^n \delta_i$, and the selected degree $m^*$.

\begin{table}[H]
\centering
\begin{tabular}{c l c c}
\hline
$x$ & Cell & $n_x$ & $\widehat{\pi}(x)$
 \\
\hline
1 & November--April, Male & 724 & 0.945 \\
2 & November--April, Female & 764 & 0.953 \\
3 & May--October, Male & 740 & 0.955 \\
4 & May--October, Female & 808 & 0.955 \\
\hline
\end{tabular}
\caption{NHANES 2017--2018 fasting-lab subsample: the four sex-by-exam-period cells used for the discrete auxiliary variable $X$, together with their cell sizes and estimated observation probabilities $\widehat{\pi}(x)$.}
\label{tab:realdata.cell.propensities}
\end{table}

\begin{table}[H]
\centering
\begin{tabular}{c c c}
\hline
$n$ & Observed rate (\%) & $m^*$
 \\
\hline
3036 & 95.2 & 516 \\
\hline
\end{tabular}
\caption{NHANES 2017--2018 real-data application (feasible estimator with estimated propensity score) on the fasting-lab subsample. The table reports the subsample size used (with $X$ observed), overall observed rate, and the LSCV-chosen Bernstein degree $m^*$.}
\label{tab:realdata.feasible}
\end{table}

Figure~\ref{fig:feasible_CDFs_nhanes_full} overlays the feasible IPW empirical CDF and its Bernstein-smoothed counterpart on the original glucose scale over the full rescaling interval. This full-range display makes clear how the CDF continues beyond the transformed value $0.40$ and reaches $1$ at the upper endpoint.

For visual detail in the lower part of the distribution, Figure~\ref{fig:feasible_CDFs_nhanes_zoom} shows a zoomed view on the transformed scale. Smoothing yields a monotone, boundary-adaptive CDF with visibly reduced roughness, consistent with the efficiency gains predicted by our theory; the selected $m^*$ balances bias and variance effectively. The workflow could be extended by refining $X$ (for example, adding age groups) or by reporting quantiles back on the original mg/dL scale via the inverse of the monotone rescaling.

\section{Summary and outlook}\label{sec:summary.outlook}

This paper developed a smooth, shape-preserving approach to CDF estimation under MAR by applying the Bernstein operator $\mathcal B_m$ to IPW empirical CDFs. The estimator is monotone and $[0,1]$-valued by construction and exploits the bounded support to mitigate boundary bias. We studied two variants, a pseudo estimator with known propensities and a feasible estimator with propensities estimated nonparametrically from discrete covariates $\bb{X}_1,\ldots,\bb{X}_n$, and we derived pointwise and integrated risk expansions, optimal degrees $m$, and asymptotic normality. A central finding was a strict variance improvement for the feasible estimator, where the asymptotic variance $\sigma^2(y)$ of the pseudo estimator is replaced by $\nu^2(y)=\sigma^2(y)-C(y)$ with $C(y)\geq 0$ (Propositions~\ref{prop:variance.tilde.F.n.m} and \ref{prop:var.hat.F.n.m}). The MSE expansions show the usual $n^{-1}$ variance and $m^{-2}$ bias together with a variance reduction term $-n^{-1} m^{-1/2}$ that rewards moderate smoothing. Our Monte Carlo results and the NHANES application indicate that smoothing improves finite-sample performance.

In practice, the workflow we advocate is straightforward. Map the response variable $Y$ to $[0,1]$ by a monotone transformation when the support is known or can be sensibly capped; estimate $\widehat{\pi}(\bb{X})$ nonparametrically when $\bb{X}$ is discrete, or by kernel methods when $\bb{X}$ is continuous; enforce a small floor $\pi_{\min}$, or trim or stabilize the weights, to control extreme IPW weights from very small propensities; select $m$ by LSCV. One LSCV evaluation has a cost of $\OO(m^2+nm)$ operations, which keeps wall-clock time modest. In our real-data analysis, smoothing added little overhead relative to the unsmoothed IPW curve and produced a visibly less ragged CDF estimate.

Our analysis has limitations that suggest future research directions. The proofs of the variance comparison in Proposition~\ref{prop:var.hat.F.n.m} were written for discrete $\bb{X}$, mainly to avoid treating the discrete and continuous covariate cases in parallel. For continuous covariates, one would instead use a flexible estimator such as the Nadaraya--Watson estimator \citep[see, e.g.,][Eq.~(4)]{MR2569798}; the same asymptotic conclusions hold, but writing out the details requires additional work. Genuinely high-dimensional covariates remain more challenging because stronger assumptions are needed to keep propensities away from~$0$. Since IPW can be unstable when some propensities are very small, simple fixes such as trimming extreme weights, using stabilized weights, or overlap weighting can be applied before smoothing, with only minor changes to the proofs. We studied a univariate response; carrying the same shape-preserving ideas to several dimensions is promising but technically harder; see, e.g., \citet{MR1293514,MR2270097,MR3474765,MR4287788,doi:10.1515/stat-2022-0111}. For inference, we proved pointwise limits. Uniform confidence bands for $F$ and for the induced quantiles will likely require multiplier or bootstrap methods tailored to the smoothed IPW empirical CDF process. It is also natural to pair Bernstein smoothing with CDF estimators that combine outcome modeling and weighting for missingness, and to accommodate survey features such as design weights, calibration, and clustering.

Overall, coupling IPW with Bernstein smoothing yields estimators that are principled, fast, and easy to implement. They respect the support geometry, adapt to boundaries automatically, and dominate unsmoothed IPW empirical CDFs for small to moderate $n$.

\begin{figure}[H]
\centering
\includegraphics[width=0.83\linewidth]{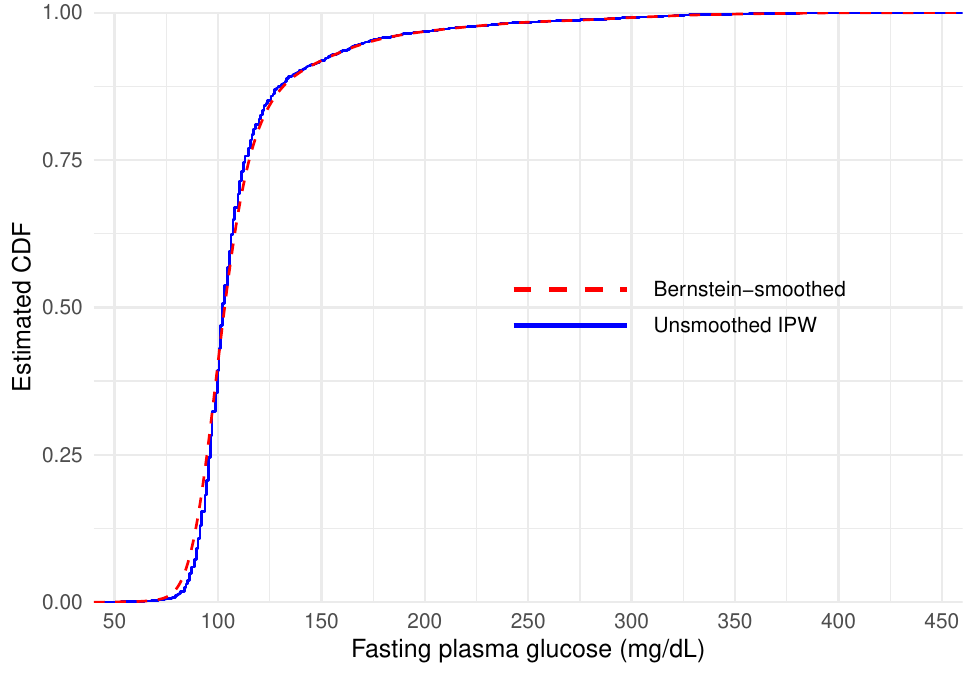}
\caption{Feasible CDF estimates of fasting plasma glucose: unsmoothed IPW versus Bernstein-smoothed (LSCV-chosen degree $m^*$), plotted on the original glucose scale. The full horizontal range $[40,460]$ mg/dL corresponds to the rescaling interval used in the Bernstein step.}
\label{fig:feasible_CDFs_nhanes_full}
\end{figure}

\begin{figure}[H]
\centering
\includegraphics[width=0.83\linewidth]{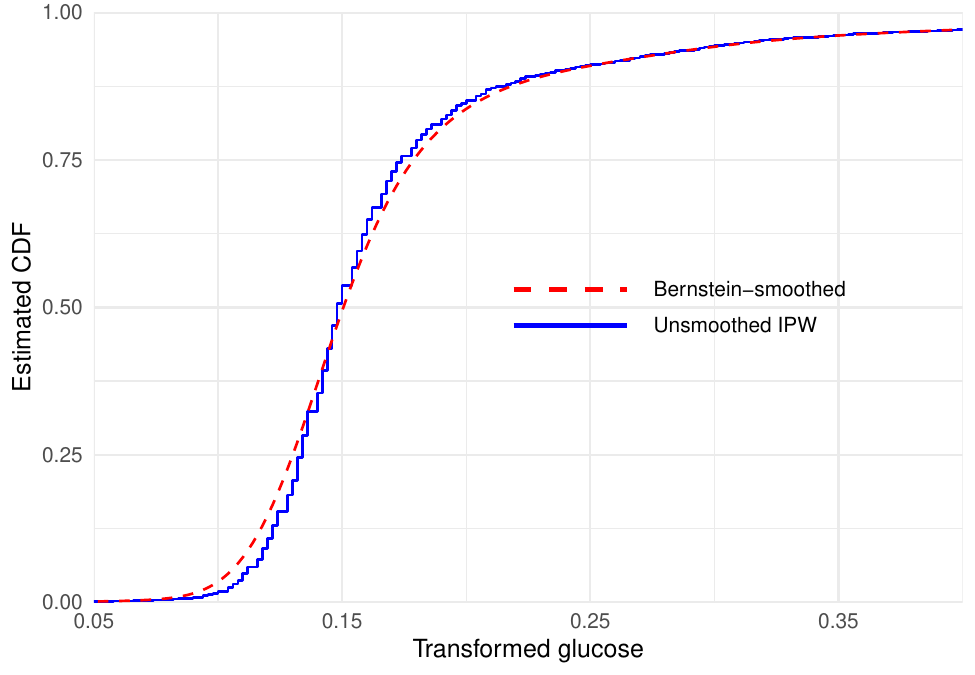}
\caption{Zoomed view of the same feasible CDF estimates on the transformed scale. The horizontal axis only shows $[0.05,0.40]$, which corresponds to approximately $[61,208]$ mg/dL under the mapping $(a,b)=(40,460)$ in \eqref{eq:mapping}.}
\label{fig:feasible_CDFs_nhanes_zoom}
\end{figure}

\section{Proofs}\label{sec:proofs}

\subsection{Proofs of the asymptotic properties of \texorpdfstring{$\smash{\widetilde{F}_{n,m}}$}{tilde{F}\_n,m}}\label{sec:proofs.asymp.tilde.F.n.m}

\begin{proof}[Proof of Proposition~\ref{prop:bias.tilde.F.n.m}]
Let $y\in (0,1)$ be given. By conditioning on $\bb{X}_1$, then using the facts that $\delta_1$ is independent of $Y_1$ conditionally on $\bb{X}_1$ and $\EE[\delta_1 \mid \bb{X}_1] = \pi(\bb{X}_1)$, we have
\begin{equation}\label{eq:prop:bias.tilde.F.n.m.first}
\begin{aligned}
\EE[\widetilde{F}_{n,m}(y)]
&= \sum_{k=0}^m \EE\left[\frac{\delta_1 }{\pi(\bb{X}_1)} \ind_{\{Y_1 \leq k/m\}}\right] b_{m,k}(y) \\
&= \sum_{k=0}^m \EE\left[\EE\left[\frac{\delta_1}{\pi(\bb{X}_1)} \ind_{\{Y_1 \leq k/m\}} \mid \bb{X}_1\right]\right] b_{m,k}(y) \\
&= \sum_{k=0}^m \EE\left[\frac{\EE(\delta_1 \mid \bb{X}_1)}{\pi(\bb{X}_1)} \EE\left[\ind_{\{Y_1 \leq k/m\}}\mid \bb{X}_1\right]\right] b_{m,k}(y) \\
&= \sum_{k=0}^m \EE\left[\ind_{\{Y_1 \leq k/m\}}\right] b_{m,k}(y) \\
&= \sum_{k=0}^m F(k/m) \, b_{m,k}(y).
\end{aligned}
\end{equation}
Next, under Assumption~\ref{ass:3}, a third-order Taylor expansion of $F$ around $y$ yields
\[
F(k/m) = F(y) + f(y) (k/m - y) + \frac{1}{2} f'(y) (k/m - y)^2 + \OO_y(|k/m - y|^3).
\]
Summing over the binomial weights, $b_{m,k}(y)$, and applying the well-known binomial moment formulas \citep[see, e.g.,][p.~24--25]{doi:10.3390/stats4010002},
\[
\begin{aligned}
\sum_{k=0}^m (k/m - y) \, b_{m,k}(y) &= \frac{1}{m} \sum_{k=0}^m (k - m y) \, b_{m,k}(y) = 0, \\
\sum_{k=0}^m (k/m - y)^2 \, b_{m,k}(y) &= \frac{1}{m^2} \sum_{k=0}^m (k - m y)^2 \, b_{m,k}(y) = \frac{y (1 - y)}{m}, \\
\sum_{k=0}^m (k/m - y)^4 \, b_{m,k}(y) &= \frac{1}{m^4} \sum_{k=0}^m (k - m y)^4 \, b_{m,k}(y) = \frac{3m^2 y^2 (y - 1)^2 + m y (1 - y) (6y^2 - 6y + 1)}{m^4},
\end{aligned}
\]
one deduces
\begin{equation}\label{eq:prop:bias.tilde.F.n.m.second}
\mathcal{B}_m(F)(y) = \sum_{k=0}^m F(k/m) \, b_{m,k}(y) = F(y) + \frac{y (1 - y)}{2m} f'(y) + \OO_y(m^{-3/2}).
\end{equation}
The error term in \eqref{eq:prop:bias.tilde.F.n.m.second} is a consequence of the Cauchy-Schwarz inequality:
\[
\begin{aligned}
\sum_{k=0}^m |k/m - y|^3 \, b_{m,k}(y)
&\leq \left(\sum_{k=0}^m (k/m - y)^4 \, b_{m,k}(y)\right)^{1/2} \left(\sum_{k=0}^m (k/m - y)^2 \, b_{m,k}(y)\right)^{1/2} \\
&= \left(\frac{3m^2 y^2 (y - 1)^2 + m y (1 - y) (6y^2 - 6y + 1)}{m^4}\right)^{1/2} \left(\frac{y (1 - y)}{m}\right)^{1/2} \\[2mm]
&= \OO(m^{-3/2}).
\end{aligned}
\]
The error term depends on $y$ in \eqref{eq:prop:bias.tilde.F.n.m.second} because $f''$ may not be bounded on $(0,1)$. This concludes the proof.
\end{proof}

\begin{proof}[Proof of Proposition~\ref{prop:variance.tilde.F.n.m}]
Let $y\in (0,1)$ be given. Because of \eqref{eq:prop:bias.tilde.F.n.m.first}, we can write
\begin{equation}\label{eq:asymp.normality.decomp}
\widetilde{F}_{n,m}(y) - \EE[\widetilde{F}_{n,m}(y)] = \widetilde{F}_{n,m}(y) - \mathcal{B}_m(F)(y) = n^{-1} \sum_{i=1}^n Z_{i,m},
\end{equation}
where
\[
Z_{i,m} = \sum_{k=0}^m \frac{\delta_i}{\pi(\bb{X}_i)} \ind_{\{Y_i \leq k/m\}} b_{m,k}(y) - \mathcal{B}_m(F)(y).
\]
Given that $Z_{1,m},\ldots,Z_{n,m}$ are i.i.d.\ and centered, and $\delta_1^2 = \delta_1$, we deduce that
\[
\begin{aligned}
\Var(\widetilde{F}_{n,m}(y))
&= n^{-1} \, \Var(Z_{1,m}) \\
&= n^{-1} \left\{\sum_{k,\ell=0}^m \EE\left[\frac{\delta_1}{\pi(\bb{X}_1)^2} \ind_{\{Y_1 \leq (k \wedge \ell)/m\}}\right] b_{m,k}(y) \, b_{m,\ell}(y)\right\} - n^{-1} \{\mathcal{B}_m(F)(y)\}^2,
\end{aligned}
\]
where $a\wedge b \leqdef \min(a,b)$. Moreover, $\delta_1$ is independent of $Y_1$ conditionally on $\bb{X}_1$, so we have
\[
\begin{aligned}
\EE\left[\frac{\delta_1}{\pi(\bb{X}_1)^2} \ind_{\{Y_1 \leq (k \wedge \ell)/m\}}\right]
&= \EE\left[\frac{\EE[\delta_1 \mid \bb{X}_1]}{\pi(\bb{X}_1)^2} \EE\left[\ind_{\{Y_1 \leq (k \wedge \ell)/m\}} \mid \bb{X}_1\right]\right] \\
&= \EE\left[\frac{1}{\pi(\bb{X}_1)} \EE\left[\ind_{\{Y_1 \leq (k \wedge \ell)/m\}} \mid \bb{X}_1\right]\right] \\
&= \EE\left[\frac{1}{\pi(\bb{X}_1)} F_{Y_1 \mid \bb{X}_1}((k \wedge \ell)/m)\right].
\end{aligned}
\]
Next, under Assumption~\ref{ass:4}, a second-order Taylor expansion of $F_{Y_1 \mid \bb{X}_1}$ around $y$ yields
\[
F_{Y_1 \mid \bb{X}_1}((k \wedge \ell)/m) = F_{Y_1 \mid \bb{X}_1}(y) + f_{Y_1 \mid \bb{X}_1}(y) ((k\wedge \ell)/m - y) + \OO_y(|(k\wedge \ell)/m - y|^2).
\]
Putting the last three equations together with \eqref{eq:prop:bias.tilde.F.n.m.second} yields
\[
\begin{aligned}
\Var(\widetilde{F}_{n,m}(y))
&= n^{-1} \, \EE\left[\frac{F_{Y_1 \mid \bb{X}_1}(y)}{\pi(\bb{X}_1)} - \{F(y)\}^2 + \frac{f_{Y_1 \mid \bb{X}_1}(y)}{\pi(\bb{X}_1)} \sum_{k,\ell=0}^m ((k\wedge \ell)/m - y) \, b_{m,k}(y) \, b_{m,\ell}(y)\right] \\
&\qquad+ \OO_y(n^{-1} m^{-1}),
\end{aligned}
\]
where the error term also implicitly uses Assumption~\ref{ass:1} to absorb $1/\pi(\bb{X}_1)$ times the Taylor series remainder in the expectation. By Lemma 4 of \citet{MR4287788}, it is known that
\[
\sum_{k,\ell=0}^m ((k\wedge \ell)/m - y) \, b_{m,k}(y) \, b_{m,\ell}(y) = m^{-1/2} \left\{-\sqrt{\frac{y (1-y)}{\pi}} + \oo_y(1)\right\}.
\]
Therefore,
\[
\begin{aligned}
\Var(\widetilde{F}_{n,m}(y))
&= n^{-1} \, \left(\EE\left[\frac{F_{Y_1 \mid \bb{X}_1}(y)}{\pi(\bb{X}_1)}\right] - \{F(y)\}^2\right) - n^{-1} m^{-1/2} \sqrt{\frac{y(1-y)}{\pi}} \EE\left[\frac{f_{Y_1 \mid \bb{X}_1}(y)}{\pi(\bb{X}_1)}\right] \\
&\quad+ \oo_y(n^{-1} m^{-1/2}).
\end{aligned}
\]
This concludes the proof.
\end{proof}

\begin{proof}[Proof of Theorem~\ref{thm:tilde.F.n.m.asymp.normality}]
Let $y\in (0,1)$. We decompose the standardized pseudo estimator as follows:
\begin{equation}\label{eq:proof.normality.decomp.tilde}
n^{1/2} \{\widetilde{F}_{n,m}(y) - F(y)\} = n^{1/2} \{\widetilde{F}_{n,m}(y) - \EE[\widetilde{F}_{n,m}(y)]\} + n^{1/2} \{\EE[\widetilde{F}_{n,m}(y)] - F(y)\}.
\end{equation}

The second term is the scaled bias. By Proposition~\ref{prop:bias.tilde.F.n.m} (under Assumptions~\ref{ass:1}~and~\ref{ass:3}),
\begin{equation}\label{eq:bias.part.tilde.F.n.m.asymp.normality}
n^{1/2} \{\EE[\widetilde{F}_{n,m}(y)] - F(y)\} = n^{1/2} \{m^{-1} B(y) + \oo_y(m^{-1})\} = (n^{1/2} m^{-1}) B(y) + \oo_y(n^{1/2} m^{-1}).
\end{equation}
This term converges to $0$ if $n^{1/2} m^{-1} \to 0$, and to $\lambda B(y)$ if $n^{1/2} m^{-1} \to \lambda$.

The first term on the right-hand side of \eqref{eq:proof.normality.decomp.tilde} is the stochastic component. As shown in the proof of Proposition~\ref{prop:variance.tilde.F.n.m} (Eq.~\eqref{eq:asymp.normality.decomp}),
\[
n^{1/2} \{\widetilde{F}_{n,m}(y) - \EE[\widetilde{F}_{n,m}(y)]\} = n^{-1/2} \sum_{i=1}^n Z_{i,m},
\]
where
\begin{equation}\label{eq:Z.i}
Z_{i,m} = \sum_{k=0}^m \frac{\delta_i}{\pi(\bb{X}_i)} \ind_{\{Y_i \leq k/m\}} b_{m,k}(y) - \mathcal{B}_m(F)(y), \quad i\in \{1,\ldots,n\},
\end{equation}
are i.i.d.\ centered random variables. We apply the Lindeberg--Feller central limit theorem for double arrays; see, e.g., \citet[Section~1.9.3]{MR0595165}. By Proposition~\ref{prop:variance.tilde.F.n.m} (under Assumptions~\ref{ass:1}, \ref{ass:3}, and \ref{ass:4}),
\begin{equation}\label{eq:asymp.norm.conv.var}
\Var(Z_{1,m}) = n \, \Var(\widetilde{F}_{n,m}(y)) \to \sigma^2(y) > 0, \quad \text{as $n\to\infty$ (and $m\to\infty$).}
\end{equation}

We verify the Lindeberg condition for $Z_{1,m}$: for every $\varepsilon > 0$,
\begin{equation}\label{eq:lindeberg.tilde}
\lim_{n\to\infty} \frac{1}{\Var(Z_{1,m})} \EE\left[ Z_{1,m}^2 \ind_{\{|Z_{1,m}|^2\, > \, \varepsilon \hspace{0.2mm} n \, \Var(Z_{1,m})\}} \right] = 0.
\end{equation}
We check whether $Z_{1,m}$ is bounded. Since $\pi(\bb{X}_i) \geq \pi_{\min} > 0$ for all $i\in \{1,\ldots,n\}$ by Assumption~\ref{ass:1}, and since $\sum_{k=0}^m b_{m,k}(y) = 1$ and $0 \leq \mathcal{B}_m(F)(y) \leq 1$, we have
\[
|Z_{1,m}| \leq \frac{1}{\pi_{\min}} \sum_{k=0}^m b_{m,k}(y) + |\mathcal{B}_m(F)(y)| \leq \frac{1}{\pi_{\min}} + 1 < \infty.
\]
Given \eqref{eq:asymp.norm.conv.var}, the indicator $\smash{\ind_{\{|Z_{1,m}|^2\, > \, \varepsilon \hspace{0.2mm} n \, \Var(Z_{1,m})\}}}$ is equal to zero almost surely for $n$ sufficiently large, and the Lindeberg condition \eqref{eq:lindeberg.tilde} is satisfied. The conclusion follows.
\end{proof}

\subsection{Proofs of the asymptotic properties of \texorpdfstring{$\smash{\widehat{F}_{n,m}}$}{hat{F}\_n,m}}\label{sec:proofs.asymp.hat.F.n.m}

\begin{proof}[Proof of Proposition~\ref{prop:bias.hat.F.n.m}]
Let $y\in (0,1)$ be given. First, note that a Taylor expansion of $1/\hat{\pi}_i(\bb{X}_{1:n})$ around $1/\pi(\bb{X}_i)$ yields
\[
\frac{1}{\hat{\pi}_i(\bb{X}_{1:n})} = \frac{1}{\pi(\bb{X}_i)} - \frac{1}{\pi(\bb{X}_i)^2} (\hat{\pi}_i(\bb{X}_{1:n}) - \pi(\bb{X}_i)) + \sum_{j=2}^{\infty} \frac{(-1)^j}{\pi(\bb{X}_i)^{j+1}} (\hat{\pi}_i(\bb{X}_{1:n}) - \pi(\bb{X}_i))^j.
\]
From \eqref{eq:estimators}, it follows that
\begin{equation}\label{eq:hat.F.n.m.vs.tilde.F.n.m}
\begin{aligned}
\widehat{F}_{n,m}(y)
&= \widetilde{F}_{n,m}(y) - n^{-1} \sum_{k=0}^m \sum_{i=1}^n \frac{\delta_i}{\pi(\bb{X}_i)^2}(\hat{\pi}_i(\bb{X}_{1:n})-\pi(\bb{X}_i)) \ind_{\{Y_i \leq k/m\}} b_{m,k}(y) \\
&\hspace{20mm}+ n^{-1} \sum_{j=2}^{\infty} \sum_{k=0}^m \sum_{i=1}^n \frac{(-1)^j \delta_i}{\pi(\bb{X}_i)^{j+1}}(\hat{\pi}_i(\bb{X}_{1:n})-\pi(\bb{X}_i))^j \ind_{\{Y_i \leq k/m\}} b_{m,k}(y).
\end{aligned}
\end{equation}
If $p$ denotes the marginal probability mass function of each $\bb{X}_i$, and
\[
\hat{p}(\bb{x}) = n^{-1} \sum_{k=1}^n \ind_{\{\bb{X}_k = \bb{x}\}}
\]
denotes the corresponding empirical estimator, then \eqref{eq:hat.pi} implies
\[
\hat{\pi}_i(\bb{X}_{1:n}) - \pi(\bb{X}_i)
= \frac{n^{-1} \sum_{j=1}^n (\delta_j - \pi(\bb{X}_i)) \, \ind_{\{\bb{X}_j = \bb{X}_i\}}}{\hat{p}(\bb{X}_i)}.
\]
A Taylor expansion of $1/\hat{p}(\bb{X}_i)$ around $1/p(\bb{X}_i)$ yields
\[
\begin{aligned}
\frac{1}{\hat{p}(\bb{X}_i)}
&= \frac{1}{p(\bb{X}_i)} - \frac{1}{p(\bb{X}_i)^2} (\hat{p}(\bb{X}_i) - p(\bb{X}_i)) + \sum_{\ell=2}^{\infty} \frac{(-1)^{\ell}}{p(\bb{X}_i)^{\ell+1}} (\hat{p}(\bb{X}_i) - p(\bb{X}_i))^{\ell}.
\end{aligned}
\]
We substitute this expansion into the previous equation and evaluate term by term:
\begin{equation}\label{eq:prop:bias.hat.F.n.m.second}
\begin{aligned}
\hat{\pi}_i(\bb{X}_{1:n}) - \pi(\bb{X}_i)
&= \frac{n^{-1} \sum_{j=1}^n (\delta_j - \pi(\bb{X}_i)) \, \ind_{\{\bb{X}_j = \bb{X}_i\}}}{p(\bb{X}_i)} \\
&\quad- \frac{n^{-1} \sum_{j=1}^n (\delta_j - \pi(\bb{X}_i)) \, \ind_{\{\bb{X}_j = \bb{X}_i\}}}{p(\bb{X}_i)^2} (\hat{p}(\bb{X}_i) - p(\bb{X}_i)) \\[-1mm]
&\quad+ n^{-1} \sum_{\ell=2}^{\infty} \sum_{j=1}^n \frac{(-1)^{\ell}}{p(\bb{X}_i)^{\ell+1}} (\delta_j - \pi(\bb{X}_i)) \, \ind_{\{\bb{X}_j = \bb{X}_i\}} (\hat{p}(\bb{X}_i) - p(\bb{X}_i))^{\ell}.
\end{aligned}
\end{equation}
The first term on the right-hand side of \eqref{eq:prop:bias.hat.F.n.m.second} is $\OO_{L^2}(n^{-1/2})$. Indeed, we have
\[
\begin{aligned}
&\EE\left[\frac{n^{-1} \sum_{j=1}^n (\delta_j - \pi(\bb{X}_i)) \, \ind_{\{\bb{X}_j = \bb{X}_i\}}}{p(\bb{X}_i)} \mid \bb{X}_{1:n}\right] \\
&\quad= \EE\left[\frac{n^{-1} \sum_{j=1}^n \EE[(\delta_j - \pi(\bb{X}_i)) \mid \bb{X}_{1:n}] \, \ind_{\{\bb{X}_j = \bb{X}_i\}}}{p(\bb{X}_i)}\right]
= 0,
\end{aligned}
\]
and
\[
\begin{aligned}
&\Var\left(\frac{n^{-1} \sum_{j=1}^n (\delta_j - \pi(\bb{X}_i)) \, \ind_{\{\bb{X}_j = \bb{X}_i\}}}{p(\bb{X}_i)} \mid \bb{X}_{1:n}\right)
= \frac{n^{-2} \sum_{j=1}^n \Var(\delta_j \mid \bb{X}_{1:n}) \, \ind_{\{\bb{X}_j = \bb{X}_i\}}}{p(\bb{X}_i)^2}
\ll \frac{n^{-1}}{p_{\min}^2},
\end{aligned}
\]
where $p_{\min} = \min_{\bb{x}} p(\bb{x}) > 0$ by Assumption~\ref{ass:2}. Hence, using the conditional variance formula,
\[
\begin{aligned}
&\Var\left(\frac{n^{-1} \sum_{j=1}^n (\delta_j - \pi(\bb{X}_i)) \, \ind_{\{\bb{X}_j = \bb{X}_i\}}}{p(\bb{X}_i)}\right) \\
&\quad= \EE\left[\Var\left(\frac{n^{-1} \sum_{j=1}^n (\delta_j - \pi(\bb{X}_i)) \, \ind_{\{\bb{X}_j = \bb{X}_i\}}}{p(\bb{X}_i)} \mid \bb{X}_{1:n}\right)\right] + 0
\ll \frac{n^{-1}}{p_{\min}^2},
\end{aligned}
\]
and thus
\[
\EE\left[\left(\frac{n^{-1} \sum_{j=1}^n (\delta_j - \pi(\bb{X}_i)) \, \ind_{\{\bb{X}_j = \bb{X}_i\}}}{p(\bb{X}_i)}\right)^2\right] \ll \frac{n^{-1}}{p_{\min}^2}.
\]
By Cauchy--Schwarz, the second moment of the second term on the right-hand side of \eqref{eq:prop:bias.hat.F.n.m.second} satisfies
\[
\begin{aligned}
&\EE\left[\left(\frac{n^{-1} \sum_{j=1}^n (\delta_j - \pi(\bb{X}_i)) \, \ind_{\{\bb{X}_j = \bb{X}_i\}}}{p(\bb{X}_i)^2} (\hat{p}(\bb{X}_i) - p(\bb{X}_i))\right)^2\right] \\
&\ll \frac{1}{p_{\min}^4} \sqrt{\EE\left[\left(n^{-1} \sum_{j=1}^n (\delta_j - \pi(\bb{X}_i)) \, \ind_{\{\bb{X}_j = \bb{X}_i\}}\right)^4\right]} \sqrt{\EE\left[(\hat{p}(\bb{X}_i) - p(\bb{X}_i))^4\right]} \\
&\ll \frac{n^{-2}}{p_{\min}^4},
\end{aligned}
\]
(each square root being $\ll n^{-1}$) so that
\begin{equation}\label{eq:tail.control.1}
\frac{n^{-1} \sum_{j=1}^n (\delta_j - \pi(\bb{X}_i)) \, \ind_{\{\bb{X}_j = \bb{X}_i\}}}{p(\bb{X}_i)^2} (\hat{p}(\bb{X}_i) - p(\bb{X}_i)) = \OO_{L^2}(n^{-1}),
\end{equation}
where the constant implicit in the error term depends on $p_{\min}$. Similarly, the residual tail sum in \eqref{eq:prop:bias.hat.F.n.m.second} satisfies
\begin{equation}\label{eq:tail.control.2}
n^{-1} \sum_{\ell=2}^{\infty} \sum_{j=1}^n \frac{(-1)^{\ell}}{p(\bb{X}_i)^{\ell+1}} (\delta_j - \pi(\bb{X}_i)) \, \ind_{\{\bb{X}_j = \bb{X}_i\}} (\hat{p}(\bb{X}_i) - p(\bb{X}_i))^{\ell} = \OO_{L^2}(n^{-1}),
\end{equation}
so that
\begin{equation}\label{eq:hat.pi.vs.pi.L2}
\EE\left[(\hat{\pi}_i(\bb{X}_{1:n}) - \pi(\bb{X}_i))^2\right] = \OO(n^{-1}).
\end{equation}
By applying H\"older's inequality to each $j$-summand in \eqref{eq:hat.F.n.m.vs.tilde.F.n.m}, one deduces that the expectation of the last term in \eqref{eq:hat.F.n.m.vs.tilde.F.n.m} satisfies
\begin{equation}\label{eq:prop:bias.hat.F.n.m.tail.estimate}
\EE\left[n^{-1} \sum_{j=2}^{\infty} \sum_{k=0}^m \sum_{i=1}^n \frac{(-1)^j \delta_i}{\pi(\bb{X}_i)^{j+1}}(\hat{\pi}_i(\bb{X}_{1:n})-\pi(\bb{X}_i))^j \ind_{\{Y_i \leq k/m\}} b_{m,k}(y)\right] = \OO(n^{-1}),
\end{equation}
where the constant implicit in the error term depends on $p_{\min}$ and $\pi_{\min}$.

It remains to evaluate the expectation of the first term on the right-hand side of \eqref{eq:hat.F.n.m.vs.tilde.F.n.m}. To do so, consider the decomposition:
\begin{equation}\label{eq:prop:bias.hat.F.n.m.third}
\begin{aligned}
&n^{-1} \sum_{k=0}^m \sum_{i=1}^n \frac{\delta_i}{\pi(\bb{X}_i)^2}(\hat{\pi}_i(\bb{X}_{1:n})-\pi(\bb{X}_i)) \ind_{\{Y_i \leq k/m\}} b_{m,k}(y) \\
&\quad= n^{-1} \sum_{k=0}^m \sum_{i=1}^n \frac{(\delta_i - \pi(\bb{X}_i))}{\pi(\bb{X}_i)^2}(\hat{\pi}_i(\bb{X}_{1:n})-\pi(\bb{X}_i)) \ind_{\{Y_i \leq k/m\}} b_{m,k}(y) \\
&\qquad+ n^{-1} \sum_{k=0}^m \sum_{i=1}^n \frac{1}{\pi(\bb{X}_i)}(\hat{\pi}_i(\bb{X}_{1:n})-\pi(\bb{X}_i)) \ind_{\{Y_i \leq k/m\}} b_{m,k}(y).
\end{aligned}
\end{equation}
For the first term on the right-hand side of \eqref{eq:prop:bias.hat.F.n.m.third}, note that, for all $i\in \{1,\ldots,n\}$, we have
\[
\EE\left[(\delta_i - \pi(\bb{X}_i))^2 \mid \bb{X}_{1:n}\right] = \Var(\delta_i \mid \bb{X}_{1:n}) = \pi(\bb{X}_i) (1 - \pi(\bb{X}_i)),
\]
and, for all $j\in \{1,\ldots,n\}\backslash\{i\}$,
\[
\begin{aligned}
&\EE\big[(\delta_i - \pi(\bb{X}_i)) (\delta_j - \pi(\bb{X}_i)) \ind_{\{\bb{X}_j = \bb{X}_i\}} \mid \bb{X}_{1:n}\big] \\
&\quad= \EE\left[(\delta_i - \pi(\bb{X}_i)) \mid \bb{X}_{1:n}\right] \EE\left[(\delta_j - \pi(\bb{X}_i)) \mid \bb{X}_{1:n}\right] \ind_{\{\bb{X}_j = \bb{X}_i\}} = 0.
\end{aligned}
\]
Hence, recalling that $\hat{\pi}_i(\bb{X}_{1:n}) - \pi(\bb{X}_i) = n^{-1} \sum_{j=1}^n (\delta_j - \pi(\bb{X}_i)) \, \ind_{\{\bb{X}_j = \bb{X}_i\}} / \hat{p}(\bb{X}_i)$, we deduce
\[
\begin{aligned}
&\EE\left[\frac{(\delta_i - \pi(\bb{X}_i))}{\pi(\bb{X}_i)^2}(\hat{\pi}_i(\bb{X}_{1:n})-\pi(\bb{X}_i)) \mid \bb{X}_{1:n}\right] \\[1mm]
&\quad= \EE\left[\frac{\EE\left[(\delta_i - \pi(\bb{X}_i))^2 \mid \bb{X}_{1:n}\right]}{n \, \pi(\bb{X}_i)^2 \hat{p}(\bb{X}_i)}\right] + \sum_{\substack{j=1 \\ j\neq i}}^n \EE\left[\frac{\EE\big[(\delta_i - \pi(\bb{X}_i)) (\delta_j - \pi(\bb{X}_i)) \ind_{\{\bb{X}_j = \bb{X}_i\}} \mid \bb{X}_{1:n}\big]}{n \, \pi(\bb{X}_i)^2 \hat{p}(\bb{X}_i)}\right] \\[-2mm]
&\quad= \EE\left[\frac{\pi(\bb{X}_i) (1 - \pi(\bb{X}_i))}{n \, \pi(\bb{X}_i)^2 \hat{p}(\bb{X}_i)}\right] + 0 \\[1mm]
&\quad= \OO(n^{-1}),
\end{aligned}
\]
where the constant implicit in the error term depends on $p_{\min}$ and $\pi_{\min}$. Given that $\delta_i$ and $Y_i$ are independent conditionally on $\bb{X}_i$, it follows that
\[
\begin{aligned}
&\EE\left[n^{-1} \sum_{k=0}^m \sum_{i=1}^n \frac{(\delta_i - \pi(\bb{X}_i))}{\pi(\bb{X}_i)^2}(\hat{\pi}_i(\bb{X}_{1:n})-\pi(\bb{X}_i)) \ind_{\{Y_i \leq k/m\}} b_{m,k}(y)\right] \\[1mm]
&\quad= \EE\left[n^{-1} \sum_{k=0}^m \sum_{i=1}^n \EE\left[\frac{(\delta_i - \pi(\bb{X}_i))}{\pi(\bb{X}_i)^2}(\hat{\pi}_i(\bb{X}_{1:n})-\pi(\bb{X}_i)) \mid \bb{X}_{1:n}\right] \EE[\ind_{\{Y_i \leq k/m\}} \mid \bb{X}_{1:n}] \, b_{m,k}(y)\right] \\[1mm]
&\quad= \OO(n^{-1}).
\end{aligned}
\]
For the second term on the right-hand side of \eqref{eq:prop:bias.hat.F.n.m.third}, note that $\EE[(\delta_j - \pi(\bb{X}_i)) \mid \bb{X}_{1:n}] \, \ind_{\{\bb{X}_j = \bb{X}_i\}} = 0$ for all $i,j\in \{1,\ldots,n\}$, so that \eqref{eq:prop:bias.hat.F.n.m.second} implies $\EE[\hat{\pi}_i(\bb{X}_{1:n})-\pi(\bb{X}_i) \mid \bb{X}_{1:n}] = 0$, and thus
\[
\begin{aligned}
&\EE\left[n^{-1} \sum_{k=0}^m \sum_{i=1}^n \frac{1}{\pi(\bb{X}_i)}(\hat{\pi}_i(\bb{X}_{1:n})-\pi(\bb{X}_i)) \ind_{\{Y_i \leq k/m\}} b_{m,k}(y)\right] \\
&\quad= \EE\left[n^{-1} \sum_{k=0}^m \sum_{i=1}^n \frac{1}{\pi(\bb{X}_i)} \EE\left[\hat{\pi}_i(\bb{X}_{1:n})-\pi(\bb{X}_i) \mid \bb{X}_{1:n}\right] \EE[\ind_{\{Y_i \leq k/m\}} \mid \bb{X}_{1:n}] \, b_{m,k}(y)\right] \\
&\quad= 0.
\end{aligned}
\]
Putting the last two equations together in \eqref{eq:prop:bias.hat.F.n.m.third} yields
\[
\EE\left[n^{-1} \sum_{k=0}^m \sum_{i=1}^n \frac{\delta_i}{\pi(\bb{X}_i)^2}(\hat{\pi}_i(\bb{X}_{1:n})-\pi(\bb{X}_i)) \ind_{\{Y_i \leq k/m\}} b_{m,k}(y)\right] = \OO(n^{-1}).
\]
Combined with \eqref{eq:prop:bias.hat.F.n.m.tail.estimate} and the decomposition \eqref{eq:hat.F.n.m.vs.tilde.F.n.m}, the conclusion follows.
\end{proof}

\begin{proof}[Proof of Proposition~\ref{prop:var.hat.F.n.m}]
Let $y\in (0,1)$ be given. We start by looking at the second term on the right-hand side of \eqref{eq:hat.F.n.m.vs.tilde.F.n.m} and we decompose it as in \eqref{eq:prop:bias.hat.F.n.m.third}. The first term on the right-hand side of \eqref{eq:prop:bias.hat.F.n.m.third} is $\OO_{L^2}(n^{-1/2})$:
\[
\begin{aligned}
&\EE\left[\left(n^{-1} \sum_{k=0}^m \sum_{i=1}^n \frac{(\delta_i - \pi(\bb{X}_i))}{\pi(\bb{X}_i)^2}(\hat{\pi}_i(\bb{X}_{1:n})-\pi(\bb{X}_i)) \ind_{\{Y_i \leq k/m\}} b_{m,k}(y)\right)^2\right] \\
&\quad\leq n^{-2} \sum_{k,\ell=0}^m \sum_{i,j=1}^n \EE\left[\left|\frac{(\delta_i - \pi(\bb{X}_i))}{\pi(\bb{X}_i)^2}(\hat{\pi}_i(\bb{X}_{1:n})-\pi(\bb{X}_i)) \ind_{\{Y_i \leq k/m\}}\right| \right. \\
&\hspace{40mm} \left.\left|\frac{(\delta_j - \pi(\bb{X}_j))}{\pi(\bb{X}_j)^2}(\hat{\pi}_j(\bb{X}_{1:n})-\pi(\bb{X}_j)) \ind_{\{Y_j \leq \ell/m\}}\right|\right] b_{m,k}(y) \, b_{m,\ell}(y) \\
&\quad\ll \frac{n^{-2}}{\pi_{\min}^4} \sum_{k,\ell=0}^m \sum_{i,j=1}^n \EE\left[\left|\hat{\pi}_i(\bb{X}_{1:n})-\pi(\bb{X}_i)\right| \left|\hat{\pi}_j(\bb{X}_{1:n})-\pi(\bb{X}_j)\right|\right] b_{m,k}(y) \, b_{m,\ell}(y) \\
&\quad\ll \frac{n^{-2}}{\pi_{\min}^4} \sum_{k,\ell=0}^m \sum_{i,j=1}^n \sqrt{\EE\left[\left|\hat{\pi}_i(\bb{X}_{1:n})-\pi(\bb{X}_i)\right|^2\right]} \sqrt{\EE\left[\left|\hat{\pi}_j(\bb{X}_{1:n})-\pi(\bb{X}_j)\right|^2\right]} b_{m,k}(y) \, b_{m,\ell}(y) \\
&\quad\ll \frac{n^{-1}}{\pi_{\min}^4},
\end{aligned}
\]
where the last bound follows from \eqref{eq:hat.pi.vs.pi.L2}. For the second term on the right-hand side of \eqref{eq:prop:bias.hat.F.n.m.third}, we have, using \eqref{eq:prop:bias.hat.F.n.m.second},
\[
\begin{aligned}
&n^{-1} \sum_{k=0}^m \sum_{i=1}^n \frac{1}{\pi(\bb{X}_i)}(\hat{\pi}_i(\bb{X}_{1:n})-\pi(\bb{X}_i)) \ind_{\{Y_i \leq k/m\}} b_{m,k}(y) \\
&\quad= n^{-2} \sum_{k=0}^m \sum_{i,j=1}^n \frac{(\delta_j - \pi(\bb{X}_i))}{\pi(\bb{X}_i)} \frac{\ind_{\{\bb{X}_j = \bb{X}_i\}}}{p(\bb{X}_i)} \ind_{\{Y_i \leq k/m\}} b_{m,k}(y) \\
&\qquad+ n^{-2} \sum_{\ell=1}^{\infty} \sum_{k=0}^m \sum_{i,j=1}^n \frac{(-1)^\ell}{p(\bb{X}_i)^{\ell+1}} \frac{(\delta_j - \pi(\bb{X}_i))}{\pi(\bb{X}_i)} \ind_{\{\bb{X}_j = \bb{X}_i\}} (\hat{p}(\bb{X}_i) - p(\bb{X}_i))^{\ell} \ind_{\{Y_i \leq k/m\}} b_{m,k}(y).
\end{aligned}
\]
By combining \eqref{eq:tail.control.1} and \eqref{eq:tail.control.2} in the proof of Proposition~\ref{prop:bias.hat.F.n.m}, the second term on the right-hand side is $\OO_{L^2}(n^{-1})$. The first term on the right-hand side is equal to
\[
n^{-1} \sum_{k=0}^m \sum_{i=1}^n \frac{(\delta_i - \pi(\bb{X}_i))}{\pi(\bb{X}_i)} F_{Y_i \mid \bb{X}_i}(k/m) \, b_{m,k}(y) + \OO_{L^2}(n^{-1}),
\]
by the law of large numbers in $L^2$. Putting all the above together shows that
\begin{equation}\label{eq:prop:var.hat.F.n.m.second}
\begin{aligned}
&n^{-1} \sum_{k=0}^m \sum_{i=1}^n \frac{\delta_i}{\pi(\bb{X}_i)^2}(\hat{\pi}_i(\bb{X}_{1:n})-\pi(\bb{X}_i)) \ind_{\{Y_i \leq k/m\}} b_{m,k}(y) \\
&\quad= n^{-1} \sum_{k=0}^m \sum_{i=1}^n \frac{(\delta_i - \pi(\bb{X}_i))}{\pi(\bb{X}_i)} F_{Y_i \mid \bb{X}_i}(k/m) \, b_{m,k}(y) + \OO_{L^2}(n^{-1}).
\end{aligned}
\end{equation}

Using the decomposition \eqref{eq:hat.F.n.m.vs.tilde.F.n.m} and the definition of the pseudo estimator $\smash{\widetilde{F}_{n,m}}$ in \eqref{eq:estimators}, we find
\begin{align}\label{eq:expansion.prop.7}
\widehat{F}_{n,m}(y)
&= \widetilde{F}_{n,m}(y) - n^{-1} \sum_{k=0}^m \sum_{i=1}^n \frac{(\delta_i - \pi(\bb{X}_i))}{\pi(\bb{X}_i)} F_{Y_i \mid \bb{X}_i}(k/m) \, b_{m,k}(y) + \OO_{L^2}(n^{-1}) \notag \\
&= n^{-1} \sum_{i=1}^n \frac{\delta_i}{\pi(\bb{X}_i)} \sum_{k=0}^m \ind_{\{Y_i \leq k/m\}} b_{m,k}(y) - n^{-1} \sum_{i=1}^n \frac{(\delta_i - \pi(\bb{X}_i))}{\pi(\bb{X}_i)} \sum_{k=0}^m F_{Y_i \mid \bb{X}_i}(k/m) \, b_{m,k}(y) \notag \\
&\quad+ \OO_{L^2}(n^{-1}) \notag \\
&\reqdef A_{n,m} - B_{n,m} + \OO_{L^2}(n^{-1}).
\end{align}
To obtain the variance of $\widehat{F}_{n,m}(y)$, it remains to compute $\Var(B_{n,m})$ and $\Cov(A_{n,m},B_{n,m})$, given that we already have the asymptotics of $\Var(A_{n,m})$ from Proposition~\ref{prop:variance.tilde.F.n.m}. Since $\delta_i$ and $Y_i$ are independent conditionally on $\bb{X}_i$, we have
\[
\begin{aligned}
\Var(B_{n,m})
&= n^{-2} \sum_{i=1}^n \Var\left(\frac{(\delta_i - \pi(\bb{X}_i))}{\pi(\bb{X}_i)} \sum_{k=0}^m F_{Y_i \mid \bb{X}_i}(k/m) \, b_{m,k}(y)\right) \\
&= n^{-2} \sum_{i=1}^n \EE\left[\Var\left(\frac{(\delta_i - \pi(\bb{X}_i))}{\pi(\bb{X}_i)} \sum_{k=0}^m F_{Y_i \mid \bb{X}_i}(k/m) \, b_{m,k}(y) \mid \bb{X}_i\right)\right] \\
&\quad+ n^{-2} \sum_{i=1}^n \Var\left(\EE\left[\frac{(\delta_i - \pi(\bb{X}_i))}{\pi(\bb{X}_i)} \sum_{k=0}^m F_{Y_i \mid \bb{X}_i}(k/m) \, b_{m,k}(y) \mid \bb{X}_i\right]\right) \\
&= n^{-2} \sum_{i=1}^n \EE\left[\frac{\Var(\delta_i \mid \bb{X}_i)}{\pi(\bb{X}_i)^2} \left(\sum_{k=0}^m F_{Y_i \mid \bb{X}_i}(k/m) \, b_{m,k}(y)\right)^2\right] \\
&\quad+ n^{-2} \sum_{i=1}^n \Var\left(\frac{(\EE[\delta_i \mid \bb{X}_i] - \pi(\bb{X}_i))}{\pi(\bb{X}_i)} \sum_{k=0}^m F_{Y_i \mid \bb{X}_i}(k/m) \, b_{m,k}(y)\right).
\end{aligned}
\]
Given that $\Var(\delta_i \mid \bb{X}_i) = \pi(\bb{X}_i) (1 - \pi(\bb{X}_i))$ and $\EE[\delta_i \mid \bb{X}_i] - \pi(\bb{X}_i) = 0$, it follows that
\[
\Var(B_{n,m}) = n^{-1} \, \EE\left[\frac{(1 - \pi(\bb{X}_1))}{\pi(\bb{X}_1)} \left(\sum_{k=0}^m F_{Y_1 \mid \bb{X}_1}(k/m) \, b_{m,k}(y)\right)^2\right].
\]
For the covariance, note that
\begin{equation}\label{eq:expectation.B.n.m.zero}
\begin{aligned}
\EE[B_{n,m}]
&= n^{-1} \sum_{i=1}^n \EE\left[\frac{(\delta_i - \pi(\bb{X}_i))}{\pi(\bb{X}_i)} \sum_{k=0}^m F_{Y_i \mid \bb{X}_i}(k/m)\right] b_{m,k}(y) \\
&= n^{-1} \sum_{i=1}^n \EE\left[\EE\left[\frac{(\delta_i - \pi(\bb{X}_i))}{\pi(\bb{X}_i)} \sum_{k=0}^m F_{Y_i \mid \bb{X}_i}(k/m) \mid \bb{X}_i\right]\right] b_{m,k}(y) \\
&= n^{-1} \sum_{i=1}^n \EE\left[\frac{(\EE[\delta_i \mid \bb{X}_i] - \pi(\bb{X}_i))}{\pi(\bb{X}_i)} \sum_{k=0}^m F_{Y_i \mid \bb{X}_i}(k/m)\right] b_{m,k}(y) \\
&= 0,
\end{aligned}
\end{equation}
and thus $\Cov(A_{n,m},B_{n,m}) = \EE[A_{n,m} B_{n,m}]$. Also, the summands of $A_{n,m}$ and $B_{n,m}$ are independent for different indices $i$, so the cross terms are zero:
\[
\EE[A_{n,m} B_{n,m}]
= n^{-2} \sum_{i=1}^n \EE\left[\frac{\delta_i (\delta_i - \pi(\bb{X}_i))}{\pi(\bb{X}_i)^2} \left(\sum_{k=0}^m \ind_{\{Y_i \leq k/m\}} b_{m,k}(y)\right) \left(\sum_{k=0}^m F_{Y_i \mid \bb{X}_i}(k/m) \, b_{m,k}(y)\right)\right].
\]
Again, the independence between $\delta_i$ and $Y_i$, conditionally on $\bb{X}_i$, yields
\[
\begin{aligned}
&\EE[A_{n,m} B_{n,m}] \\
&= n^{-2} \sum_{i=1}^n \EE\left[\frac{\EE[\delta_i (\delta_i - \pi(\bb{X}_i)) \mid \bb{X}_i]}{\pi(\bb{X}_i)^2} \EE\left[\sum_{k=0}^m \ind_{\{Y_i \leq k/m\}} b_{m,k}(y) \mid \bb{X}_i\right] \left(\sum_{k=0}^m F_{Y_i \mid \bb{X}_i}(k/m) \, b_{m,k}(y)\right)\right] \\
&= n^{-2} \sum_{i=1}^n \EE\left[\frac{\EE[\delta_i (\delta_i - \pi(\bb{X}_i)) \mid \bb{X}_i]}{\pi(\bb{X}_i)^2} \left(\sum_{k=0}^m F_{Y_i \mid \bb{X}_i}(k/m) \, b_{m,k}(y)\right)^2\right].
\end{aligned}
\]
Since $\delta_i^2 = \delta_i$ and $\EE[\delta_i \mid \bb{X}_i] = \pi(\bb{X}_i)$, we have
\[
\begin{aligned}
\EE[A_{n,m} B_{n,m}]
&= n^{-2} \sum_{i=1}^n \EE\left[\frac{\pi(\bb{X}_i) (1 - \pi(\bb{X}_i))}{\pi(\bb{X}_i)^2} \left(\sum_{k=0}^m F_{Y_i \mid \bb{X}_i}(k/m) \, b_{m,k}(y)\right)^2\right] \\
&= n^{-1} \, \EE\left[\frac{(1 - \pi(\bb{X}_1))}{\pi(\bb{X}_1)} \left(\sum_{k=0}^m F_{Y_1 \mid \bb{X}_1}(k/m) \, b_{m,k}(y)\right)^2\right] \\
&= \Var(B_{n,m}).
\end{aligned}
\]
It follows that
\begin{equation}\label{eq:prop:var.hat.F.n.m.end}
\begin{aligned}
\Var(\widehat{F}_{n,m}(y))
&= \Var(A_{n,m}) - 2 \, \EE[A_{n,m} B_{n,m}] + \Var(B_{n,m}) \\[2mm]
&= \Var(\widetilde{F}_{n,m}(y)) - \Var(B_{n,m}) \\
&= \Var(\widetilde{F}_{n,m}(y)) - n^{-1} \, \EE\left[\frac{(1 - \pi(\bb{X}_1))}{\pi(\bb{X}_1)} \left(\sum_{k=0}^m F_{Y_1 \mid \bb{X}_1}(k/m) \, b_{m,k}(y)\right)^2\right].
\end{aligned}
\end{equation}
Finally, by applying the Taylor expansion under Assumption~\ref{ass:4},
\begin{equation}\label{eq:Taylor.2}
F_{Y_i \mid \bb{X}_i}(k/m) = F_{Y_i \mid \bb{X}_i}(y) + f_{Y_i \mid \bb{X}_i}(y) (k/m - y) + \OO_y(|k/m - y|^2),
\end{equation}
we obtain
\begin{equation}\label{eq:near.end}
\begin{aligned}
&n^{-1} \, \EE\left[\frac{(1 - \pi(\bb{X}_1))}{\pi(\bb{X}_1)} \left(\sum_{k=0}^m F_{Y_1 \mid \bb{X}_1}(k/m) \, b_{m,k}(y)\right)^2\right] \\
&\quad= n^{-1} \, \EE\left[\frac{(1 - \pi(\bb{X}_1))}{\pi(\bb{X}_1)} \{F_{Y_1 \mid \bb{X}_1}(y)\}^2\right] \\
&\qquad- 2 n^{-1} \, \EE\left[\frac{(1 - \pi(\bb{X}_1))}{\pi(\bb{X}_1)} F_{Y_1 \mid \bb{X}_1}(y) f_{Y_1 \mid \bb{X}_1}(y) \sum_{k=0}^m (k/m - y) \, b_{m,k}(y)\right] \\
&\qquad+ \OO_y\left(\frac{n^{-1}}{\pi_{\min}} \left(\sum_{k=0}^m (k/m - y) \, b_{m,k}(y)\right)^2\right).
\end{aligned}
\end{equation}
The second term on the right-hand side is zero because $\sum_{k=0}^m (k - m y) \, b_{m,k}(y) = 0$. The error term is $\OO_y(n^{-1} m^{-1})$ by Jensen's inequality since
\[
\left(\sum_{k=0}^m (k/m - y) \, b_{m,k}(y)\right)^2 \leq \sum_{k=0}^m (k/m - y)^2 \, b_{m,k}(y) = m^{-2} \sum_{k=0}^m (k - m y)^2 \, b_{m,k}(y) = m^{-1} y (1 - y).
\]
The error term depends on $y$ in \eqref{eq:near.end} because $f_{Y_1 \mid \bb{X}_1}'$ may not be bounded on $(0,1)$. Therefore,
\[
\begin{aligned}
n^{-1} \, \EE\left[\frac{(1 - \pi(\bb{X}_1))}{\pi(\bb{X}_1)} \left(\sum_{k=0}^m F_{Y_1 \mid \bb{X}_1}(k/m) \, b_{m,k}(y)\right)^2\right]
&= n^{-1} \, \EE\left[\frac{(1 - \pi(\bb{X}_1))}{\pi(\bb{X}_1)} \{F_{Y_1 \mid \bb{X}_1}(y)\}^2\right] \\
&\quad+ \OO_y(n^{-1} m^{-1}).
\end{aligned}
\]
Substituting this expression back into \eqref{eq:prop:var.hat.F.n.m.end} yields the conclusion.
\end{proof}

\begin{proof}[Proof of Theorem~\ref{thm:hat.F.n.m.asymp.normality}]
Let $y\in (0,1)$ be given. We utilize the asymptotic representation \eqref{eq:expansion.prop.7} derived in the proof of Proposition~\ref{prop:var.hat.F.n.m}:
\[
\widehat{F}_{n,m}(y) = \widetilde{F}_{n,m}(y) - B_{n,m} + \OO_{L^2}(n^{-1}),
\]
where $B_{n,m}$ is defined as
\begin{equation}\label{eq:U.i}
B_{n,m} = n^{-1} \sum_{i=1}^n U_{i,m}, \qquad U_{i,m} = \frac{(\delta_i - \pi(\bb{X}_i))}{\pi(\bb{X}_i)} \sum_{k=0}^m F_{Y_i \mid \bb{X}_i}(k/m) \, b_{m,k}(y).
\end{equation}
The expression for the standardized estimator now follows:
\[
n^{1/2} \{\widehat{F}_{n,m}(y) - F(y)\} = n^{1/2} \{\widetilde{F}_{n,m}(y) - B_{n,m} - F(y)\} + \OO_{L^2}(n^{-1/2}).
\]

For the main term, we know $\EE[\widetilde{F}_{n,m}(y)] = \mathcal{B}_m(F)(y)$ by \eqref{eq:prop:bias.tilde.F.n.m.first} and $\EE[B_{n,m}] = 0$ by \eqref{eq:expectation.B.n.m.zero}. We decompose it as:
\[
n^{1/2} \{\widetilde{F}_{n,m}(y) - B_{n,m} - F(y)\} = n^{1/2} \{\widetilde{F}_{n,m}(y) - \mathcal{B}_m(F)(y) - B_{n,m}\} + n^{1/2} \{\mathcal{B}_m(F)(y) - F(y)\}.
\]
The second term is the bias term, analyzed in the proof of Theorem~\ref{thm:tilde.F.n.m.asymp.normality} (Eq.~\eqref{eq:bias.part.tilde.F.n.m.asymp.normality}). The first term is the stochastic part. Let $W_{i,m} = Z_{i,m} - U_{i,m}$, where $Z_{i,m}$ is defined in \eqref{eq:Z.i} and $U_{i,m}$ in \eqref{eq:U.i}. Then
\[
n^{1/2} \{\widetilde{F}_{n,m}(y) - \mathcal{B}_m(F)(y) - B_{n,m}\} = n^{-1/2} \sum_{i=1}^n W_{i,m}.
\]
The $W_{i,m}$'s are i.i.d.\ and centered, and $\Var(W_{1,m}) \to \nu^2(y) > 0$ as $n\to\infty$ by Proposition~\ref{prop:var.hat.F.n.m}.

We verify the Lindeberg condition for $W_{1,m}$: for every $\varepsilon>0$,
\begin{equation}\label{eq:Lindeberg.2}
\lim_{n\to\infty} \frac{1}{\Var(W_{1,m})} \EE\left[ W_{1,m}^2 \ind_{\{|W_{1,m}|^2\, > \, \varepsilon \hspace{0.2mm} n \, \Var(W_{1,m})\}}\right] = 0.
\end{equation}
By Assumption~\ref{ass:1}, $\pi(\bb{X}_i)\geq \pi_{\min}>0$. Using $\sum_{k=0}^m b_{m,k}(y)=1$ and $0\leq \mathcal{B}_m(F)(y)\leq 1$,
\[
|Z_{1,m}|\leq \pi_{\min}^{-1}+1, \qquad
|U_{1,m}|\leq \pi_{\min}^{-1},
\]
so $|W_{1,m}|\leq 2\pi_{\min}^{-1}+1 < \infty$. Since $\Var(W_{1,m})\to \nu^2(y)>0$, we have $\varepsilon \hspace{0.2mm} n \, \Var(W_{1,m})\to \infty$, hence for all sufficiently large $n$, the indicator $\smash{\ind_{\{|W_{1,m}|^2\, > \, \varepsilon \hspace{0.2mm} n \, \Var(W_{1,m})\}}}$ is equal to zero almost surely. Therefore the expectation in \eqref{eq:Lindeberg.2} is eventually $0$, and the Lindeberg condition holds.
\end{proof}

\appendix

\begin{appendices}

\section{Reproducibility}\label{app:code}

The \textsf{R} code that generated the figures, the simulation study results, and the real-data application is available online in the GitHub repository of \citet{GharbiJedidiKhardaniOuimet2026github}.

\section{List of acronyms}\label{app:acronyms}

\begin{tabular}{llll}
&BISE &\hspace{20mm} &boundary integrated squared error \\
&CDF &\hspace{20mm} &cumulative distribution function \\
&i.i.d. &\hspace{20mm} &independent and identically distributed \\
&I-IPW &\hspace{20mm} &integrated inverse probability weighted \\
&IPW &\hspace{20mm} &inverse probability weighting/weighted \\
&ISE &\hspace{20mm} &integrated squared error \\
&KDE &\hspace{20mm} &kernel density estimator \\
&LSCV &\hspace{20mm} &least-squares cross-validation \\
&MAR &\hspace{20mm} &missing at random \\
&MISE &\hspace{20mm} &mean integrated squared error \\
&MSE &\hspace{20mm} &mean squared error \\
&NHANES &\hspace{20mm} &National Health and Nutrition Examination Survey \\
&NMAR &\hspace{20mm} &not missing at random (nonignorable) \\
\end{tabular}

\end{appendices}

\section*{Acknowledgments}
\addcontentsline{toc}{section}{Acknowledgments}

The authors thank the three referees for their careful reading and constructive comments, which led to substantial improvements in the clarity and quality of the paper.

\section*{Funding}
\addcontentsline{toc}{section}{Funding}

The work of Wissem Jedidi is supported by the Ongoing Research Funding program (ORF-2026-162) at King Saud University, Riyadh, Saudi Arabia. Fr\'ed\'eric Ouimet is supported by the start-up fund (1729971) from the Universit\'e du Qu\'ebec à Trois-Rivi\`eres.

\addcontentsline{toc}{section}{References}
\setlength{\bibsep}{3pt plus 0ex}

\bibliographystyle{authordate1}
\bibliography{bib}

\end{document}